\documentclass{amsart}

\usepackage{amsmath,amssymb,amsrefs}
\usepackage{enumerate}
\usepackage{tikz}
\usetikzlibrary{calc}
\usetikzlibrary{arrows,positioning}

\newcommand{\udots}{\mathinner{\mskip1mu\raise1pt\vbox{\kern7pt\hbox{.}}
  \mskip2mu\raise4pt\hbox{.}\mskip2mu\raise7pt\hbox{.}\mskip1mu}}

\usepackage{algorithm}
\usepackage[noend]{algpseudocode}

\algblock[TryCatchFinally]{try}{endtry}
\algcblock[TryCatchFinally]{TryCatchFinally}{finally}{endtry}
\algcblockdefx[TryCatchFinally]{TryCatchFinally}{catch}{endtry}
	[1]{\textbf{catch} #1}
	{\textbf{end try}}

\makeatletter
\def\BState{\State\hskip-\ALG@thistlm}
\makeatother
\numberwithin{equation}{section}
\newtheorem{thm}[equation]{Theorem}
\newtheorem*{thm*}{Theorem}

\newtheorem{lem}[equation]{Lemma}
\newtheorem{prop}[equation]{Proposition}
\newtheorem{defn}[equation]{Definition}

\newenvironment{problem}{\medskip\begin{minipage}{0.9\textwidth}
\noindent\rule{\textwidth}{1pt} \smallskip }
{\smallskip \rule{\textwidth}{1pt}
\end{minipage}\medskip}

\theoremstyle{remark}
\newtheorem*{remark*}{Remark}

\renewcommand{\leq}{\leqslant}
\newcommand{\bmto}{\rightarrowtail}

\DeclareMathOperator{\gr}{gr}
\DeclareMathOperator{\Aut}{Aut}
\DeclareMathOperator{\Atp}{Aut}

\DeclareMathOperator{\Isom}{Isom}
\DeclareMathOperator{\Adj}{Adj}
\DeclareMathOperator{\im}{{\rm im}}
\DeclareMathOperator{\GL}{GL}
\DeclareMathOperator{\End}{End}
\DeclareMathOperator{\Hom}{Hom}
\newcommand{\pseudo}{\Psi\hspace*{-1mm}\Isom}

\renewcommand{\geq}{\geqslant}
\renewcommand{\leq}{\leqslant}

\title{Testing isomorphism of graded algebras}
\author{Peter A.\ Brooksbank}
\address{Brooksbank, Department of Mathematics, Bucknell University, Lewisburg, PA 17837, USA}
\author{E.\ A.\ O'Brien}
\address{O'Brien, Department of Mathematics, University of Auckland, Private Bag 92019, Auckland,
New Zealand}
\author{James B.\ Wilson}
\address{Wilson, Department of Mathematics, Colorado State University, Fort Collins, CO 80523, USA} 
\thanks{This work was supported in part by the Marsden Fund of
New Zealand via grant UOA 1626, by NSF grants DMS-1620454 and DMS-1620362,
and by the Simons Foundation $\#$281435. We thank the referee for helpful comments.}

\begin{document}

\begin{abstract}
We present a new algorithm to decide isomorphism between
finite graded algebras. For a broad class of nilpotent Lie algebras, 
we demonstrate that it runs in time polynomial in the order of the 
input algebras. 
We introduce heuristics that often dramatically
improve the performance of the algorithm 
and report on an implementation in {\sc Magma}.
\end{abstract}

\maketitle

%%%%%%%
\section{Introduction}
\label{sec:intro}
It is possible to decide if algebraic objects $A$ and $B$ of order 
$n$ are isomorphic by fixing a generating sequence $a_1,\dots,a_d$ for $A$ and 
searching through all sequences $b_1,\dots,b_d$ in $B$ until we find an 
identification $a_i\mapsto b_i$ that extends to an isomorphism
$A\to B$. This process takes $n^d$ steps: for groups and algebras, where $d$ can be 
as large as $\log n$, the resulting complexity is not polynomial in the orders of the input objects.
Despite significant progress over the years on various isomorphism problems, 
asymptotic improvements over ``brute force" for substantial classes of objects are rare. 

We introduce a general strategy for testing isomorphism of 
finite graded algebras.
It is particularly effective for nilpotent matrix Lie algebras, and 
we describe a class of such algebras for which our isomorphism test
runs in time polynomial  in the order of the input algebra.  
We have also implemented a version in {\sc Magma} \cite{magma}.

While graded algebras are natural structures in their own right, 
they also arise from the study of other algebraic structures.
For example, given a ring $R$, one can compute its Jacobson radical, $J$,
and consider the graded algebra $\gr R=R/J\oplus \bigoplus_{i=1}^{\infty} J^i/J^{i+1}$.
Similarly, the intersection $O_p(G)$
of the Sylow $p$-subgroups of a finite group $G$ for a prime $p$ dividing $|G|$ is 
normal and nilpotent, so  $\gr G=\mathbb{Z}[G/O_p(G)]\oplus \bigoplus_{i=1}^{\infty} \eta_i/\eta_{i+1}$ is a graded Lie $\Bbb{F}_p$-algebra,
where $\eta_1:=O_p(G)$ and $\eta_{i+1}=[\eta_{i},\eta_1]\eta_i^p$.
Isomorphism tests for associated graded structures work with
individual graded components, often exploiting the power of linear algebra.

Existing uses of graded algebras within isomorphism testing proceed
sequentially through the grading; see~\cite{eick}, for example. 
Starting with the first graded component, 
one considers all possible isomorphisms between corresponding components,
and uses the graded product to decide which of them induces an
isomorphism between subsequent components. While this iterative approach usually
offers improvements over brute force, a single large 
homogeneous component may create a bottleneck.  
Our approach is not constrained by the need
to process the components sequentially.  Instead, it identifies
sections of the two graded algebras between which the list
of possible maps is small and can be computed quickly.
It then determines which of these maps between sections lift to 
isomorphisms of the algebras.
\smallskip

To state our main result, we require a few preliminaries. 
Let $K$ be a finite field. A {\em $K$-algebra} is a 
$K$-module, $A$, equipped with a (possibly nonassociative) $K$-bilinear product 
$\circ\colon A\times A\bmto A$. If, as a $K$-module, 
\[
A=\bigoplus_{s=0}^{\infty} A_s,~~~~~~\mbox{where}~~A_s\circ A_t\leq A_{s+t}, 
\]
then $A$ is {\em $\mathbb{N}$-graded}.  
We assume that $A$ is {\em generated in degree} 1 in the sense that,
for all $s>0$, 
$A_s=\sum_{j=1}^{s-1}A_j\circ A_{s-j}$.  An isomorphism
between graded algebras that maps each graded component of one
algebra to the corresponding component of the other is 
a {\em graded isomorphism}.
For each $s\geq 1$, one can restrict the product on $A$ to produce
a {\em bilinear map} $A_1\times A_s\bmto A_{s+1}$. The ring of
{\em adjoints} of this bilinear map, denoted $\mathcal{M}_s$, is 
the largest ring $R$ faithfully represented on
$A_1\oplus A_s$ such that the bilinear map $A_1\times A_s\bmto A_{s+1}$ 
factors through the tensor product space $A_1\otimes_{R} A_s$. 
Defined explicitly as operators in~\eqref{eq:adjoints}, 
$\mathcal{M}_s$ has group of units $\mathcal{M}_s^{\times}$.
\smallskip

Our main result, proved in Section~\ref{sec:iso}, is the following.

%%%
\begin{thm}
\label{thm:main}
There is a deterministic algorithm that, given two finite graded $K$-algebras $A$ and $B$
generated in degree $1$, decides whether or not there is a graded 
isomorphism $A\to B$.
The algorithm has complexity
\[
O(\min_s\{|\mathcal{M}_{s}^{\times}|\cdot|\Aut(A_{s+1})|\}\cdot
(\dim A)^{2\omega} \cdot \log^2 |K|),
\]
where $s$ runs over the grading and $2\leq \omega<3$ is the exponent of matrix 
multiplication over $K$.
\end{thm}
%%%

By comparison, the sequential approach has complexity
$\tilde{O}(|K|^{(\dim A)^2})$. While the estimate in Theorem~\ref{thm:main}
is not asymptotically better in all cases, the flexibility to
choose which component to process first may lead to dramatic improvements. 
As one illustrative example, define $\mathcal{F}_q(d_1,\dots,d_{\ell})$ to be the smallest class of 
graded Lie algebras containing the Lie subalgebras 
${\frak L}_*\leq \mathfrak{gl}_{d_1+\ldots+d_{\ell}}(q)$ satisfying 
\begin{align*}%\label{eq:dense-intro}
 	[{\frak L}_*,{\frak L}_*] & = \left\{
 	\begin{bmatrix} 
 	0_{d_1} & 0_{d_1,d_2} & * & \cdots & * \\
 	& \ddots & \ddots & \ddots & \vdots \\
	& & \ddots & \ddots & * \\
 	& &  & 0_{d_{\ell-1}} & 0_{d_{\ell-1},d_{\ell}} \\
 	& & & & 0_{d_{\ell}}
 	\end{bmatrix}\right\}
 	\leq {\frak L}_* \leq 
 	\left\{
 	\begin{bmatrix} 
 	0_{d_1} & * & \cdots & * \\
 	& \ddots & \ddots & \vdots \\
 	& & \ddots & * \\
 	&	& & 0_{d_{\ell}} 
 	\end{bmatrix}\right\},
\end{align*}
where $0_{m,n}$ denotes the zero $(m\times n)$-matrix, and $0_n=0_{n,n}$.
(A {\em class of algebras}
is a collection closed under isomorphism; in particular, no specific 
representation is assumed as part of the input.)
An analysis of the algorithm in Theorem~\ref{thm:main} applied to 
$\mathcal{F}(d_1,\dots,d_{\ell})$ establishes the 
following. 

%%%
\begin{thm}
\label{thm:Lie-main}
Isomorphism testing in $\mathcal{F}_q(d_1,\dots,d_{\ell})$ is in time $\tilde{O}(|{\frak L}|^{m^2/\varepsilon})$, where 
\begin{align*}
	\varepsilon & = 
		\sum_{1\leq i<j\leq \ell} d_i d_j, &
	m & = \min_{s\leq 1+\frac{d}{2}}\left\{\sum_{i=1}^{\ell-s} d_i d_{i+s}\right\}.
\end{align*}
The number of isomorphism classes in $\mathcal{F}(d_1,\dots,d_{\ell})$ is
$q^{O((d_1+\cdots+d_{\ell})^2)}$.
\end{thm}
The complexity in Theorem~\ref{thm:Lie-main}
is polynomial in the size of the algebras when both 
$m^2, \varepsilon \in O((d_1+\ldots+d_{\ell})^2)$, a condition that holds,
for instance, when $d_i$ is bounded, and in many other cases.
\smallskip

One of the motivations of this work is to develop practical tools to 
decide isomorphism within families of finite groups and algebras. 
From this viewpoint, we are concerned with developing algorithms 
that perform well as a function of the length of 
standard encodings of the input 
algebras, such as by generating sets, or as bases with structures constants. 
Despite the significant improvements offered by the 
strategy underlying our main results, we have encountered 
situations where the necessary exhaustive search is still intractably large. 

To address this concern, in Section~\ref{sec:local} we introduce heuristics
that use local invariants to deduce global restrictions on the 
possible automorphisms arising from $A_1\times A_s\bmto A_{s+1}$, thereby reducing 
substantially the ensuing exhaustive search. More precisely, we design a labeling of the 
points and lines of the projective  space on $A_{s+1}$ that is invariant under isomorphism. 
The labeling offers sufficient variability that the resulting constraints often reduce 
intractable searches to the practical realm. 
We also revisit another ``local-to-global" process called {\em fingerprinting} 
that was introduced in~\cite{ELGOB}. In Section \ref{sec:imp} we report on our implementation 
of the algorithms in {\sc Magma}, and show for some families of examples
that our techniques have significant practical impact.

We prove a more general version of 
Theorem~\ref{thm:main} in which algebras are not  necessarily generated in degree 1, 
and may be graded using an arbitrary monoid. 
This general treatment allows our algorithms to take as input more refined gradings
on the given algebras that often decompose them into smaller pieces. 
The development of such refined gradings is an emerging area of study that
may lead to faster isomorphism tests;  see~\cite{Maglione:filters}, for example. 

%%%%%
%%%%%
\section{Preliminaries}
\label{sec:prelim}
Throughout the paper $K$ denotes a finite field, and all $K$-vector spaces are 
finite-dimensional. The $K$-dual of a $K$-vector space $V$ is denoted by 
$V^{\dagger}$. If $f\colon U\to V$ is a linear map, then 
$f^{\dagger}\colon V^{\dagger}\to U^{\dagger}$ denotes the dual map.  

%%%
\subsection{Graded algebras}
\label{subsec:graded}
For convenience in our exposition we assume that an 
algebra $A$ is specified by a basis $\{a_1,\dots,a_d\}$ together
with {\em structure constants}  
$[\alpha_{ij}^k]$ defined by
\begin{align*}
	a_i\cdot a_j & = \sum_k \alpha_{ij}^k a_k.
\end{align*}
Although the structure constant model of an algebra is both less compact and less efficient than 
alternatives---such as matrix algebras specified by generating sets, multivariate polynomials, 
and more general quotients of free algebras---it provides a convenient,  uniform starting point.

Let $M$ be a commutative
monoid with pre-order $\prec$, where $a\prec b$ if there exists $c \in M$ such that 
$a+c=b$. We assume that, relative to $\prec$, every nonempty subset has a 
minimal element.  We also assume that $0$ is a minimal element of $M$ so that 
$M$ is conical. This ensures that we can perform (Noetherian) induction on the 
indices in $M$: if $S\subset M$ has the property that for every $s\in S$ there exists 
$t\in S\setminus\{s\}$ with $t\prec s$, then $S=\emptyset$. The
monoids $M=\mathbb{N}^c$ satisfy the necessary conditions.

An algebra $A_*$ is {\em $M$-graded} if $A_*=\bigoplus_{s\in M} A_s$ and for all
$s,t\in M$, $A_s\circ A_t\subset A_{s+t}$.  We assume, for each $s$, that
$A_s$ is the $K$-linear span of $A_s\cap \{a_1,\dots,a_d\}$. 
We say that $A$ is {\em generated in degrees $T$} if
\begin{align}\label{eq:gen-degree}
	&(\forall s) & s\in M\setminus T & ~~\Longrightarrow
	\sum_{{\tiny \begin{array}{c} s=s_1+s_2, \\ s_1,s_2\notin 
	\{0,s\} \end{array}}} A_{s_1}\circ A_{s_2}~~=~A_s.
\end{align}
As $M$ is conical, if $A_*$ is generated in degrees $T$, then $0\in T$.  
The following observation is a direct consequence of the induction principle.

\begin{lem}
If $A_*$ is an $M$-graded algebra generated in degrees $T$ then, for each
$s\in M$, $A_s$ consists of linear combinations of products of
elements in $\bigcup_{t\in T}A_t$.
\end{lem}

If $M=\mathbb{N}$, then the generating
degrees are $0$ and $1$. Here, $A_0$ contributes no shift in the 
grading; it is customary to ignore $0$ and regard $A_*$ as 
generated in degree $1$.  The main results in the introduction were formulated 
for this special case.

As an infinite monoid can grade a finite algebra, repetitions 
$A_s=A_t$ for $s\neq t$ are common.  To avoid redundancy,
one can truncate all rays in the conical monoid $M$ in the first place
where all are stable.  This is always possible but may produce a less familiar 
monoid: for each $c\in\mathbb{N}$, associate to $M$ the $\ell$-truncated 
cyclic monoid $\mathbb{N}_{\ell}=\{0,\ldots,\ell\}$, where
$x\boxplus y:=x+y$ if $x+y<\ell$, and $x\boxplus y :=\ell$ otherwise.

%%%
\begin{prop}
\label{prop:small-filter}
Let $M=\langle m_1,\ldots,m_d\rangle$, 
let $A_*$ be an $M$-graded algebra, and let $\ell$ be the smallest positive integer satisfying
\begin{align*}
\forall i\in\{1,\ldots,d\},~\forall u\in M\setminus\{0\} & & A_{\ell m_i}=A_{\ell m_i+u}.
\end{align*} 
Then $A_*$ is naturally 
$(\mathbb{N}_{\ell})^d$--graded where 
$A_{(r_1,\ldots,r_d)}=A_{r_1 m_1+\cdots+r_d m_d}$.
In particular, we may assume for fixed $d$ that $|M|\in O((\dim A_*)^{d})$. 
\end{prop}

%%%%%%%%%
\subsection{Bimaps}
\label{sec:bimaps}
A {\em bi-additive map} (or just {\em bimap})  is a tuple $U_*=\langle U_2,U_1,U_0,\circ\rangle$ 
where the $U_i$ are abelian groups and $\circ\colon U_2\times U_1\bmto U_0$ 
is a function satisfying the following
two-sided distributive law:
\begin{align*}
 	 (u_2+u'_2)\circ u_1  & = u_2\circ u_1 + u'_2\circ u_1 &
	 u_2\circ(u_1+u'_1)  & = u_2\circ u_1 + u_2\circ u'_1 .
\end{align*}
We always work with $K$-bimaps: the $U_i$ are $K$-modules and $\circ$ is $K$-bilinear.
The {\em dimension} of $U_*$ is the sum of the dimensions of its component spaces, namely
\begin{align*}
	\dim U_* := \dim U_2+\dim U_1+\dim U_0.
\end{align*}
A {\em homotopism} $f_*\colon U_*\to V_*$ is a
tuple $(f_i\colon U_i\to V_i \;|\; i\in \{0,1,2\})$ satisfying 
\begin{align}\label{def:homotopism}
	&(\forall u_2)(\forall u_1) 
		& (u_2\circ u_1)f_0 & = u_2f_2\circ u_1f_1.
\end{align} 
Denote by $\Hom(U_*,V_*)$ the set of all homotopisms $f_*\colon U_*\to V_*$.

The class of bimaps with homotopisms forms 
the {\em homotopism category}. There are
various natural morphisms on classes of bimaps, such as 
adjoint-morphisms \cite{Wilson:division},
so we name the categories after the morphisms rather than the objects.
We are interested primarily in {\em isotopisms}, namely homotopisms 
whose constituent maps 
are all isomorphisms. The {\em autotopism group} of a bimap 
$U_*$ is
\begin{equation}
\label{eq:aut}
	\Aut(U_*)   =  \Hom(U_*,U_*) \cap \prod_{i=0}^2 \Aut(U_i).
\end{equation}

%%%%%%
\subsection{Shuffling bimaps}
\label{sec:shuffles}
Since a bimap $U_*=\langle U_2,U_1,U_0,\circ\rangle$ consists of multiple
components, it is tedious, both in proofs and in algorithms, to 
specify individual components.
Accordingly, we often ``shuffle" indices in our bimaps.  
To do this, however, we must 
ensure that autotopisms are unaffected by the process, so we now
define precisely what we mean by shuffling. 
Related categorical subtleties are considered 
in \cite{Wilson:division}; 
observe that our treatment differs from that of~\cite{IQ}.

Given a bimap $U_*=\langle U_2,U_1,U_0,\circ\rangle$ and a permutation 
$\sigma\in {\rm Sym}(\{2,1,0\})$, we define a new bimap $U_*^{\sigma}$. 
It suffices to consider just
transpositions, as these generate all possible permutations,
but the definition depends on the particular transposition. 
First consider $\sigma=(2,1)$, the transposition swapping $2$ and $1$.  
Set $U^{\sigma}_2:=U_1$, $U^{\sigma}_1:=U_2$, 
$U^{\sigma}_0:=U_0$, and define $\circ^{\sigma}\colon U_2^{\sigma}
\times U_1^{\sigma}\bmto U_0^{\sigma}$ as follows: given $u_2\in 
U_2^{\sigma}$ and $u_1\in U_1^{\sigma}$, set
\begin{align*}
	 u_2\circ^{\sigma} u_1 &:= u_1\circ u_2.
\end{align*}
We define $U_*^{\sigma}$ to be $\langle U_2^{\sigma},U_1^{\sigma},U_0^{\sigma},\circ^{\sigma}\rangle$.
A homotopism $f_*\colon U_*\to V_*$ is sent to $f^{\sigma}_*\colon U^{\sigma}_*\to 
V^{\sigma}_*$ where $f^{\sigma}_2 := f_1$, $f^{\sigma}_1:=f_2$,
$f_0^{\sigma}:=f_0$.  In the literature this is also called 
the {\em transpose} of a bimap.

For $\sigma=(1,0)$ and $\sigma=(2,0)$ we need a slightly more elaborate construction.  
We consider just $(1,0)$ since $(2,0)$ works in the same way.
Set $U^{\sigma}_2:=U_2$, $U^{\sigma}_1:=U_{0}^{\dagger}$, 
$U^{\sigma}_0:=U_{1}^{\dagger}$.  We switch to Greek letters when
working with elements in a dual vector space.  Define 
$\circ^{\sigma}:U_2^{\sigma}\times U_1^{\sigma}\bmto U_0^{\sigma}$ as 
follows: given $u_2\in U_2^{\sigma}$ and $\nu_1\in U_1^{\sigma}$ (so $\nu_1\colon U_{1\sigma}\to K$), 
define $(u_2\circ^{\sigma}\nu_1)=\nu_0\in U_0^{\sigma}=\Hom(U_1,K)$ 
as follows:
\begin{align*}
	\nu_0\colon u_{0\sigma}\mapsto  (u_2 \circ u_{0\sigma})\nu_1.
\end{align*}
A complication arises with homotopisms, since interchanging $f_1$
and $f_0$ is no longer meaningful.  
Instead, we put $f_2^{\sigma}:=f_2$, $f_1^{\sigma}:=f_0^{\dagger}$ and $f_0^{\sigma}:=f_1^{\dagger}$.
In particular, $f_*^{\sigma}$ is not a homotopism but satisfies
the following condition:
\begin{align*}
	&(\forall u_2)(\forall \nu_1) 
	& (u_2\in U_2^{\sigma})\wedge (\nu_1\in V^{\sigma}_1) & \Rightarrow
	u_2f_2^{\sigma} \circ^{\sigma} \nu_1
	= (u_2\circ^{\sigma} \nu_1 f_1^{\sigma})f_0^{\sigma}.
\end{align*}
This illustrates the inherent delicacy of the otherwise trivial idea
of re-indexing variables.  In the special case where $f_*$ is an isotopism
we observe that
$\hat{f}_*^{\sigma}=(f_2,(f_0^{\dagger})^{-1}, (f_1^{\dagger})^{-1})$ is 
an isotopism $U_*^{\sigma}\to V_*^{\sigma}$.  

%%%%%
\subsection{Weakly Hermitian bimaps}
\label{sec:weakly}
A bimap is {\em weakly Hermitian} if it is isotopic to its transpose, namely to
its $\sigma=(2,1)$ shuffle.  The
familiar symmetric and alternating bimaps are examples of weakly Hermitian 
bimaps, but the notion is substantially more general.  
Associated to weakly Hermitian bimaps are group invariants that respect their 
symmetry.  If we fix an isotopism $g_*:U_*\to U_*^{\dagger}$, then 
\begin{align*}
	\Psi{\rm Isom}(U_*) & = \{ f_* \in {\rm Aut}(U_*)\colon f_* g_*=g_*f_*^{\dagger}\}
\end{align*}
is the {\em group of pseudo-isometries} of $U_*$ and
does not depend on the choice of $g_*$.

%%%%%
\subsection{Algebras that operate on bimaps}
\label{sec:bimap-op}
Consider the following algebras determined by a bimap $U_*$.
\begin{align}
	\mathcal{L}(U_*) & = \{ (f^{\dagger},h^{\dagger})\in 
		{\rm End}(U_2)\times {\rm End}(U_0) \colon
		(f^{\dagger}u_2)\circ u_1=h^{\dagger}(u_2\circ u_1)\}, \\
\label{eq:adjoints}
	\mathcal{M}(U_*) & = \{ (f,g^{\dagger})\in 
		{\rm End}(U_2)\times {\rm End}(U_1) \colon
		(u_2f)\circ u_1=u_2\circ(g^{\dagger}u_1)\}, \\	
	\mathcal{R}(U_*) & = \{ (g,h)\in 
		{\rm End}(U_1)\times {\rm End}(U_0) \colon
		u_2\circ (u_1g)=(u_2\circ u_1)h\},~\mbox{and}\\
	\mathcal{T}(U_*) & = \mathcal{L}(U_*)\oplus \mathcal{M}(U_*)
		\oplus \mathcal{R}(U_*).
\end{align}
Each $U_i$ is a natural module under each of these 
rings, and is thus a $\mathcal{T}(U_*)$-module. Although the 
action is non-unital---for example, the representation of $\mathcal{L}(U_*)$ 
on $U_1$ is trivial---it is more convenient to think of each of 
$U_2,U_1,U_0$ as a module over a common ring than continually to clarify that
the action on $U_1$ is by $\mathcal{T}(U_*)/\mathcal{L}(U_*)$, and so on. 
We also use the following algebra:
\begin{equation}
	\mathcal{C}(U_*)  = \left\{ f_*\in \prod_{i=0}^2{\rm End}(U_i) : 
		  (u_2 f_2)\circ u_1 = u_2\circ (u_1f_1) = (u_2\circ u_1)f_0\right\}.
\end{equation}

We need one final notion.  Fix bimaps $U_*$ and $V_*$.  As
defined in \cite{Wilson:division}, an {\em adjoint-morphism} 
$(f,g)\colon U_*\to V_*$ is a pair of maps $f\colon U_2\to V_2$ and 
$g\colon V_1\to U_1$ satisfying
\begin{align*}
	&(\forall u_2)(\forall v_1)
	& (u_2\in U_2)\wedge (v_1\in V_1)\Rightarrow
		(u_2f)\circ v_1 & = u_2\circ (v_1 g).
\end{align*}
The set of all such pairs $(f,g)$ is denoted $\Adj(U_*,V_*)$.
This defines another category on bilinear maps distinct 
from those already discussed: it is an abelian category
and plays a role similar to modules of 
rings \cite{Wilson:division}*{Theorem~2.27}. 
In particular, observe that ${\rm Adj}(U_*,U_*)$ is simply the ring
$\mathcal{M}(U_*)$.  As adjoint-bimap categories are 
{\em not} equivalent to module categories 
\cite{Wilson:division}*{Theorem~2.10}, however,
we must adapt some established results in module theory
to suit our purpose.

%%%%%
\section{Testing isotopism of bimaps}
\label{sec:autotopism}
In this section we consider the problem of deciding 
if two bimaps are equivalent under isotopisms.  
This is an essential step
in our isomorphism test for graded algebras, but it is also a problem of
independent interest. We require an efficient
solution to the following problem.

\begin{problem}
{\sc IsotopismCoset}
\begin{description}
\item[Given]  $K$-bilinear maps $U_*$ and $V_*$.

\item[Return] the coset ${\rm Iso}(U_*,V_*)$ of isotopisms $U_*\to V_*$.
\end{description}
\end{problem}

\noindent Here, we present a basic algorithm to solve this problem;
in Section~\ref{sec:local} we introduce heuristics to speed up the construction.

Note, if $f_*\colon U_*\to V_*$ is an arbitrary isotopism, then  
${\rm Iso}(U_*,V_*)=\Aut(U_*)f_*$, so  the output can be encoded
compactly using generators for $\Aut(U_*)$.  If $U_*=V_*$, 
then the output is simply $\Aut(U_*)$.

%%%%%%
\subsection{Principal autotopisms}
The difficulty of {\sc IsotopismCoset} 
stems from having to find solutions to quadratic polynomials in 
multiple variables: namely,  
we solve for $(f_2,f_1,f_0)$ where the parameters $f_2$ and $f_1$
occur in a product $u_2 f_2\circ u_1 f_1$.  Quadratic varieties are as complex
as arbitrary varieties, 
but fixing any one of the $f_i$ renders the problem substantially more tractable.
Thus, we consider first the following restricted version of the autotopism group problem;
in Section~\ref{sec:extend} we handle the coset version.

\begin{problem}
{\sc PrincipalAutotopismGroup}
\begin{description}
\item[Given] a $K$-bilinear map $U_*$ and $i\in \{2,1,0\}$.

\item[Return] generators for 
$\Aut(U_*)^{(i)}  :=\{ f_*\in \Atp(U_*)\colon  f_i=1_{U_i}\}$.
\end{description}
\end{problem}

The following observation leads to an efficient solution to this problem.
\begin{prop}
\label{prop:principal}
For a bimap $U_*$, the following hold:
\begin{enumerate}[(i)]
\item $\Aut(U_*)^{(2)}\cong \mathcal{R}(U_*)^{\times}$;
\item $\Aut(U_*)^{(1)}\cong \mathcal{L}(U_*)^{\times}$; and
\item $\Aut(U_*)^{(0)}\cong \mathcal{M}(U_*)^{\times}$.
\end{enumerate}
\end{prop}

\begin{proof} 
Following our discussion in Section~\ref{sec:shuffles}, we can 
assume $i=0$ after a possible shuffling of the variables.  (We
stress once more that reindexing requires some re-adjustment of the
resulting isotopisms.) If $f_*\in \Aut(U_*)$ and
$f_0=1$, then
\begin{align*}
	(u_2 f_2)\circ u_1 & = (u_2f_2) \circ (u_1 f_1^{-1} f_1)
		 = (u_2\circ (u_1 f_1^{-1}))f_0
		 =u_2\circ (u_1 f_1^{-1}),
\end{align*}
so $(f_2,f_1^{-1})\in \mathcal{M}(U_*)$.  If $(f_2,f_1)\in 
\mathcal{M}(U_*)$ then $(f_2,f_1^{-1},1_{U_0})\in {\rm Aut}(U_*)$.
\end{proof}

\begin{algorithm}
\caption{Principal Autotopism Group}\label{algo:principal}
\begin{algorithmic}[1]
\Require a $K$-bilinear map $U_*$ and $i\in \{2,1,0\}$.
\Ensure generators for $\Aut(U_*)^{(i)}$.

\State Choose a permutation $\sigma$ on $\{2,1,0\}$ with $i\sigma=0$.
\State Solve a system of linear equations to find a basis for 
$\mathcal{M}(U_*^{\sigma})$.
\State Use \cite[Theorem 2.3]{BO} to compute generators $X$ for the group of
units of $\mathcal{M}(U_*^{\sigma})$.
\State Set $G=\langle (f,g^{-1},1)^{\sigma} : (f,g)\in X\rangle\leq \prod_{i=0}^2\Aut(U_i)$.
\State \Return $G$.
\end{algorithmic}
\end{algorithm}

\begin{prop}
\label{prop:principal-algo}
Algorithm~$\ref{algo:principal}$ solves {\sc PrincipalAutotopismGroup}.
As a deterministic algorithm it runs in 
time $O((\dim U_*)^{2\omega}\log^2 |K|+{\rm char}~K)$;
a Las Vegas variant runs in time $O((\dim U_*)^{2\omega}\log^2 |K|)$.
\end{prop}

\begin{proof}
Since the correctness of the algorithm is clear from Proposition~\ref{prop:principal}
and the mechanics of shuffling variables, we focus on the complexity.
Line 2 involves solving a system of 
$(\dim U_2)(\dim U_1)(\dim U_0)$ linear equations in $(\dim U_2)^2+(\dim U_1)^2$
variables, which can be done in time $O((\dim U_*)^{2\omega}\log^2 |K|)$.
Line 3 invokes the algorithm of \cite{BO}*{Theorem~2.3}, which
depends on the ability to factor polynomials over $K$.  
The algorithm runs in
Las Vegas polynomial-time $O((\dim A)^{2\omega}\log^2 |K|)$ if we use
Las Vegas polynomial factorization routines such as that of \cite{cantor-zassenhaus}.
A deterministic algorithm is known when the ground field
of $K$ can be listed:  in this case Line 3 runs 
in time $O((\dim A)^{2\omega}\log^2 |K|+{\rm char}~K)$.  The remaining steps of 
Algorithm~\ref{algo:principal} have negligible influence on the timing,
so the result follows.
\end{proof}

%%%%%
\subsection{Extending to isotopisms}
\label{sec:extend}
Our next objective is to solve a single instance of isotopism.   
We focus first on principal isotopisms, and 
assume that $i=0$ by shuffling coordinates.

\begin{problem}
{\sc PrincipalIsotopism}
\begin{description}
\item[Given] $K$-bilinear maps  $U_*$ and $V_*$ and a map $f_i\colon U_i\to V_i$ for 
fixed $i \in \{2,1,0\}$.

\item[Return]  an isotopism $f_*\colon U_*\to V_*$ extending $f_i$.
\end{description}
\end{problem}

Just as the construction of the principal autotopism group is a problem
in rings, the construction of a principal isotopism resembles a problem
in modules. As we indicated in Section \ref{sec:prelim}, however,
it is not precisely a module problem that we solve.  

\begin{defn}
An {\em orthogonal decomposition} of a bimap $U_*$ is
a pair of direct decompositions $U_2=\bigoplus_j U_{2j}$ and 
$U_1=\bigoplus_k U_{1k}$ such that $U_{2j}\circ U_{1k}=0$ if $j\neq k$.  
Each $U_{ij}$ is an {\em orthogonal factor}.
\end{defn}

For example, if $U_*=\langle K^2,K^3,K,\circ\rangle$ where 
\begin{align*}
	u_2\circ u_1 & = u_2\begin{bmatrix} 1 & 0 & 0\\ 0 & 1 & 0 \end{bmatrix}u_1^{\dagger}
\end{align*}
then $U_2=K(1,0)\oplus K(0,1)$ and $U_1=K(1,0,0)\oplus K(0,1,0)\oplus K(0,0,1)$ is an orthogonal decomposition.  
More generally, in terms of structure constants this implies that  
$A^{(k)}=[a_{ij}^k]$ is block diagonal, the blocks coinciding with the
$(U_{2j},U_{1j})$ pairs; see \cite{Wilson:division}*{Section~2.4} for details.
\smallskip

Our plan is to imitate the algorithm of \cite{BL:mod-iso}, which 
builds a module isomorphism one direct summand at a time.  
Both that algorithm and our adaptation rely on the following
construction. If $X\subseteq \End_K(V)$, then let $\overline{X}$
be the semigroup generated by $X$ and let $K\langle X\rangle$ be 
the $K$-linear span of $\overline{X}$. (Contrary to the usual notion of  
enveloping algebra, $K\langle X\rangle$
need not be unital: this occurs if, and only if, the identity 
can be written as a linear combination of elements of $\overline{X}$.)

%%%
\begin{thm}[\cite{BL:mod-iso}*{Corollary 2.5}]
\label{thm:find-non-nil}
There is a polynomial-time algorithm that, given 
$X\subseteq \End_K(V)$, decides if  $K\langle X\rangle$
is nilpotent and, if not, returns a product of elements
in $\overline{X}$ that is not nilpotent. 
\end{thm}
%%%

The algorithm in~\cite{BL:mod-iso} uses non-nilpotent elements
to decompose the modules into direct summands.
Instead of module isomorphisms we construct principal isotopisms; 
instead of direct summands we use orthogonal factors.   
We capture the key recursive step with the following technical definition.

\begin{defn}
Fix $f_0\colon U_0\to V_0$.
A {\em partial $f_0$-isotopism} of bimaps $U_*$ and $V_*$ is an isotopism 
$g_*=(g_2,g_1,g_0)$ defined on the restriction to some orthogonal factors of 
$U_*$ and $V_*$, and such that $g_0=f_0$.  A 
partial $f_0$-isotopism is {\em maximal} if it is not a restriction to 
proper subspaces of another partial $f_0$-isotopism.
\end{defn}
 
The idea is to build a (possibly nilpotent) ring from 
two sets of adjoint-morphisms.  If this ring contains an invertible
element, then we find the desired 
principal isotopism.  
Following \cite{BL:mod-iso}, we propose Algorithm \ref{alg-2}
to construct a maximal partial isotopism.

\begin{algorithm}[!htbp]
\caption{Partial Principal Isotopism}\label{algo:partial-isotopism}
\begin{algorithmic}[1]
\Require bimaps $U_*$ and $V_*$, 
and an isomorphism $f_0\colon U_0\to V_0$.
\Ensure a maximal partial $f_0$-isotopism.

\State $\mathcal{X}\gets {\rm Basis}({\rm Adj}(U_*,V_*^{f_0}))$; 
$\mathcal{Y}\gets {\rm Basis}({\rm Adj}(V_*^{f_0},U_*))$.
\State $A \gets K\langle xy: x\in \mathcal{X}, y\in \mathcal{Y}\rangle
\subset \End(U_2)\times \End(U_1)^{{\rm op}}$.
\State For $i=1,2$, $U_i^-\gets U_i$;\quad $U_i^+\gets 0$;\quad $f_i\gets 0$.
\While{$A$ has $z=xyw$ not nilpotent with $x\in\mathcal{X}$ and $y\in\mathcal{Y}$}
	\State Find $n\geq 0$ such that, for $i=1,2$, $U_i^-=\ker z^n\oplus \im z^n$.
	\State $U_i^-\gets \ker z^n$;\quad $U_i^+\gets U_i^+\oplus \im z^n$; \quad
		$f_i \gets f_i \oplus {\rm res}_{\im z^n}(x)$.
	\State Restrict $A$ to $\ker z^n$.
\EndWhile
\State \Return $f_*=(f_2\colon U_2^+\to V_2^+, f_1\colon U_1^+\to V_1^+, f_0\colon U_0\to V_0)$.
\end{algorithmic} \label{alg-2}
\end{algorithm}

%%%
\begin{prop}
\label{thm:partial}
Algorithm~$\ref{algo:partial-isotopism}$ is deterministic and constructs a maximal $f_0$-isotopism
in polynomial time  $O((\dim U_*)^{2\omega} \log^2 |K|)$.
\end{prop}

\begin{proof}
The objective of the algorithm is to find an invertible element of 
$\Adj(U_*, V_*^{f_0})$.  To do this, we first create a (possibly 
non-unital) algebra $A$ in Line 2 by composing the sets of homomorphisms
created in Line $1$.  Observe that composition in the second variable is in the 
op-ring $\End(U_1)^{{\rm op}}$.    

First, note that $A\subset \mathcal{M}(U_*)$.  Secondly, if $f_0$ extends to an 
isotopism $f_*\colon U_*\to V_*$ then $(f_2,f_1^{-1})\in \Adj(U_*, V_*^{f_0})$ and 
$(f_1,f_2^{-1})\in \Adj(V_*^{f_0}, U_*)$; in particular, $A$ 
contains units.  However, finding a unit of $A$ does not guarantee that we can 
extract an invertible element of $\Adj(U_*, V_*^{f_0})$.  

We claim that the loop starting in Line~4 maintains the following invariants: 
for $i\in\{2,1\}$, 
$U_i=U_i^+\oplus U_i^-$ is an orthogonal decomposition of $U_i$;
and $(f_2,f_1,f_0)$ is a partial $f_0$-isotopism.
Clearly, this is true at the start. 

By its construction in Line 4, clearly
$z\in\mathcal{M}(U_*)$, and so $z^n\in\mathcal{M}(U_*)$.  Note, for every 
$b\in\mathcal{M}(U_*)$, the decomposition $U_i = \ker b\oplus \im b$ is
orthogonal.  Therefore we maintain throughout an orthogonal
decomposition $U_i=U_i^+\oplus U_i^-$.  Furthermore, by Fitting's lemma, 
$z^n$ is invertible on $\im z^n$.  The guard of the loop in Line 4 is a call to
Theorem~\ref{thm:find-non-nil}, which provides
$x\in {\rm Adj}(U_*,V_*^{f_0})$ such that $z=xyw$.  
As $z^n$ is invertible on $\im z^n$,
and $z^n=x\cdots$, it follows that $x$ is injective on $\im z^n$.  
Since all spaces are finite, this injection is a bijection.  
Therefore the extension of  $f_i$ by the restriction of $x$ to the 
image of $z^n$ remains a partial $f_0$-isotopism, as required.

Finally, the loop continues while $A$ contains a non-nilpotent element.
Thus, the partial $f_0$-isotopism is maximal and the output is correct.

The major work is solving the system of linear equations in Line~1;
this results in the complexity stated in the theorem.
\end{proof}

\begin{remark*}
In many settings, invertible elements of $\Adj(U_*,V_*^{f_0})$ may be 
found by random search with high probability, but there are examples 
that require an exponential number of samples to return an invertible
element.  Nevertheless, 
once $\Adj(U_*,V_*^{f_0})$ is constructed, it is 
sensible to test a small number of random elements.
\end{remark*}

The following is now immediate.

%%%%%
\begin{thm}
\label{thm:principal}
There is a deterministic, polynomial-time algorithm to solve {\sc PrincipalIsotopism}.
\end{thm}
%%%%%

We are finally ready to present Algorithm~\ref{algo:isotopism-coset}, the 
main result of this section.

\begin{algorithm}[!htbp]
\caption{Isotopism coset}\label{algo:isotopism-coset}
\begin{algorithmic}[1]
\Require bimaps $U_*$ and $V_*$.
\Ensure the coset ${\rm Iso}(U_*,V_*)$ of isotopisms $U_*\to V_*$.
\State Choose a permutation $\sigma$ on $\{2,1,0\}$ with $\dim U_{0\sigma}$ minimized. 
\If{$\dim U_{0\sigma}\neq \dim V_{0\sigma}$} \Return $\emptyset$.
\EndIf
\State Choose any isomorphism $f_{0\sigma}:U_{0\sigma}\to V_{0\sigma}$.
\State Choose $G$ such that $\Aut(U_*)|_{U_{0\sigma}}\leq G\leq \Aut(U_{0\sigma})$.\hfill 
/*{\em\;see Section~$\ref{sec:local}$}\,*/
\State $I\gets \emptyset$.
\ForAll{$g\in G$}
\State 
/*{\em\; Algorithms $\ref{algo:principal}$ \& $\ref{algo:partial-isotopism}$} */
\State Find the coset $C$ of isotopisms 
$h_*\colon U_*\to V_*$ with $h_{0\sigma}=g f_{0\sigma}$.  
\State $I\gets I\cup C$.
\EndFor
\State \Return $I$
\end{algorithmic}
\end{algorithm}

Viewing the group $G$ in Line 4 of this algorithm 
as a parameter, the following is an immediate consequence of the results 
and algorithms of this section.

%%%%%%%%
\begin{thm}
\label{thm:PIC}
Algorithm~$\ref{algo:isotopism-coset}$ solves 
{\sc IsotopismCoset}
and runs deterministically in time
$O(|G|\max_i\{(\dim U_i)^{2\omega}\}\log^2 |K|)$, where $G$
is the group in Line~$4$.  
\end{thm}

%%%%%%%%%
%%%%%%%%%
\section{Testing isomorphism of graded algebras}
\label{sec:iso}
Our algorithm to decide isomorphism between graded algebras proceeds
under the assumption that an 
isomorphism exists.  If this does not occur then the test is 
aborted.  A standard mechanism to do this is to raise an exception. This
means that all further steps are aborted and the algorithm backtracks to 
the nearest place that can handle the exception.  
Recall from Proposition~\ref{prop:small-filter} that
we assume an algebra $A_*$ is graded by a monoid $M$ that has size polynomial
in $\dim A_*$.

%%%%
\subsection{Extending isotopisms to graded isomorphisms}
The isomorphism algorithm  proceeds by attempting to extend
isotopisms between the bimaps obtained by restricting the given products to certain fixed homogeneous components.
We therefore begin by considering the necessary extension problem; our solution is summarized in
Algorithm~\ref{algo:extend}. As $A_*$ and $B_*$ are
generated in degrees $T$, we assume that we have a homotopism whose restriction to $T$, namely
$f_T=\{(f_t\colon A_t\to B_t)\colon  t\in T\}$, is defined. Our task is to extend $f_T$ to a graded
algebra homomorphism $f_*\colon A_*\to B_*$.

For $s\in M \setminus T$, define
\begin{align*}
	A_{\otimes s} &~~ := \bigoplus_{{\tiny \begin{array}{c} s=s_1+s_2, \\ s_i\notin \{0,s\} \end{array}}} A_{s_1}\otimes A_{s_2}~~
	\subset~~ A_*\otimes A_*.
\end{align*}
Given $s\in M$, we can construct $\{(s_1,s_2)\in M\times M: s=s_1+s_2\}$ 
in at most $|M|^2\leq (\dim A_*)^2$ steps.  Define
\begin{align*}
	f_{\otimes s} & ~~:= \bigoplus_{{\tiny \begin{array}{c} s=s_1+s_2, \\ s_i\notin \{0,s\} \end{array}}} f_{s_1}\otimes f_{s_2}
	\in \Hom(A_{\otimes s},B_{\otimes s}),
\end{align*}
where $f_{s_1}\otimes f_{s_2}$ is defined component-wise on $A_{s_1}\otimes A_{s_2}$.
For $s\in M$, let $\prec\hspace*{-0.1cm}s=\{u\in M: u\prec s\}$. 
For $R\subset M$, write $R\prec s$ if 
$r\prec s$ for every $r\in R$.  If $A_*$ is generated in degrees $T$, then 
setting $s=\sum_{t\in T} t$ implies that $T\prec s$. 

\begin{algorithm}[!htbp]
\caption{Extending Homotopisms}
\label{algo:extend}
\begin{algorithmic}[1]
\Require finite $M$-graded $K$-algebras $A_*$ and $B_*$ 
generated in degrees $T \subset M$, and $f_T=\{(f_t\colon A_t\to B_t)\colon  t\in T\}$ such that
for every $t,r\in T$ with $t+r\in T$, the triple $(f_t,f_r, f_{t+r})$ is a homotopism
from $A_r\times A_t\bmto A_{r+t}$ to $B_r\times B_t\bmto B_{r+t}$.
\Ensure an algebra homomorphism $f_*\colon A_*\to B_*$
extending $f_T$, or raise an exception if $f_T$ does not extend to an algebra
homomorphism.
\State $R\gets T$.
\While{$R\neq M$}
	\State Choose $s\in M \setminus R$ such that 
	$\{r\in M\colon r\prec s, r\neq s\}\subset R$.
	\State Induce $\pi_{s,A_*}\colon A_{\otimes s}\to A_{s}$ and likewise 
	$\pi_{s,B_*}\colon B_{\otimes s}\to B_s$.
	\State Compute $\ker\pi_{s,A_*}$ and $\ker\pi_{s,B_*}$.
	\If{$(\ker \pi_{s,A_*})f_{\otimes s}\leq \ker \pi_{s,B_*}$}
		\State $f_s\gets \pi_{s,A_*}^{-1} \cdot f_{\otimes s} \cdot \pi_{s,B_*}\in \Hom(A_s, B_s)$.
		\State $R\gets R\cup\{s\}$.
	\Else
		\State {\bf raise exception} no extension exists at $s$.
	\EndIf
\EndWhile
\State \Return $f_*$.
\end{algorithmic}
\end{algorithm}

%%%
\begin{prop}
Algorithm~$\ref{algo:extend}$ is correct and runs deterministically in time 
\[
O(|M|({\rm dim} A_*)^{2\omega}\log^2 |K|).
\] 
\end{prop}
%%%

\begin{proof}
Consider first the correctness of the algorithm.  Observe  that 
$f_*\colon A_*\to B_*$ is a graded homomorphism if, and only if, for all $s$ and $t$,
$(f_s,f_t,f_{s+t})$ is a homotopism from the 
bimap $A_s\times A_t\bmto A_{s+t}$ to $B_s\times B_t\bmto B_{s+t}$.

We claim that the loop starting on Line 2 has the following invariants:
\begin{enumerate}[(i)] 
\item $R$ is an interval closed set of indices; 
\item for all $r\in R$, the map $f_r\colon A_r\to B_r$ is defined; and 
\item for all $r,r'\in R$, if $r+r'\in R$, then $(f_r,f_{r'},f_{r+r'})$
is an isotopism from $A_r\times A_{r'}\bmto A_{r+r'}$ to  $B_r\times B_{r'}\bmto B_{r+r'}$. 
\end{enumerate}
The loop terminates when
$R=M$, so $f_*\colon A_*\to B_*$ is defined and is 
consequently a graded algebra
isomorphism.

Consider any $s$ selected in Line 3. Since $A_*$ is generated in degrees 
$T$, and $T\subset R$, it follows that $\pi_{s,A_*}$ is surjective and 
\begin{align*}
	A_s\cong A_{\otimes s}/\ker \pi_{s,A_*}.
\end{align*}
The same holds for $B_*$.  By our choice of $s$, 
for every $s_1,s_2\notin \{0,s\}$ with $s=s_1+s_2$, $f_{s_1}$ and $f_{s_2}$ are defined,
so $f_{\otimes s}$ is defined.  If 
$(\ker\pi_{s,A_*})f_{\otimes s}\leq \ker \pi_{s,B_*}$, then we may induce
$f_s\colon A_s\to B_s$ on the generators of $A_s$ as follows:
for $a_{s_1}\in A_{s_1}$ and $a_{s_2}\in A_{s_2}$, 
\begin{align*}
	(a_{s_1}\circ a_{s_2})f_s & := (a_{s_1}\otimes a_{s_2})f_{\otimes s}	
	\equiv a_{s_1} f_{s_1}\circ a_{s_2}f_{s_2}\bmod{\ker\pi_{s,B_*}}.
\end{align*}
Conversely, if $f_{\prec s}$ extends to a graded isomorphism 
$f_*\colon A_*\to B_*$, then $f_s$ is defined as above.  Thus, if 
$(\ker\pi_{s,A_*})f_{\otimes s}\not\leq \ker \pi_{s,B_*}$ we
conclude that $\{f_t\colon t\in T\}$ does not extend, and raise an exception to
abort all subsequent steps.

Next, we analyze the timing. The loop executes at most $|M|\in (\dim A_*)^{O(1)}$
iterations, and the timing in each one is dominated by the computation of 
$\ker \pi_{s,A_*}$ and $\ker \pi_{s,B_*}$ and the subsequent membership test
in the latter.  Each requires solving systems of linear
equations in $(\dim A_s)^2$ variables.  In total this takes
$O(\sum_{s} (\dim A_s)^{2\omega}\log^2 |K|)$ steps, as stated.
\end{proof}

%%%
\subsection{The isomorphism test}
For an $M$-graded algebra $A_*$ and $S\subset M$, define 
$A_S=\bigoplus_{s\in S} A_s$; for a graded homomorphism $f_*$,
define $f_S=\bigoplus_{s\in S} f_s$.
Algorithm~\ref{algo:simple-iso-graded} is our
isomorphism test for graded algebras. The mechanism for selecting $S$ in Line 1 
is discussed in Section~\ref{sec:selection}.

%%%
\begin{algorithm}[!htbp]
\caption{Graded Isomorphism Coset}
\label{algo:simple-iso-graded}
\begin{algorithmic}[1]
\Require finite $M$-graded algebras $A_*$ and $B_*$
generated in degrees $T\subset M$.
\Ensure the coset of graded isomorphisms $A_*\to B_*$, or $\emptyset$
if $A\not\cong B$. 
%---------
\State Choose $\emptyset\neq S\subset M$. \hfill 
/*{\em\;see Section~$\ref{sec:selection}$}\,*/
\State Restrict multiplication in 
$A_*$ to obtain the bimap $A_T\times A_S\bmto A_{S+T}$.
\State Similarly, obtain the bimap $B_T\times B_S\bmto B_{S+T}$.
\State 
$I\gets {\rm Iso}(A_T\times A_S\bmto A_{S+T}, B_T\times B_S\bmto B_{S+T})$, using Algorithm~$\ref{algo:isotopism-coset}$.
\State $\Gamma\gets \{ f_*:A_*\to B_*: f_*$ extends some $(f_T,f_S,f_{S+T})\in I\}$, using
Algorithm~\ref{algo:extend}.
\State \Return $\Gamma$.
\end{algorithmic}
\end{algorithm}

%%%
\begin{prop}
Algorithm~$\ref{algo:simple-iso-graded}$ is correct.
\end{prop}

\begin{proof}
If $f_*\colon A_*\to B_*$ is a graded isomorphism,
then the restriction 
\[
\left(\bigoplus_{t\in T}f_{t},~~\bigoplus_{s\in S} f_s,
\bigoplus_{s\in S,t\in T} f_{s+t}\right)
\] 
is an isotopism from 
$A_T\times A_S\bmto A_{S+T}$ to  $B_T\times B_S\bmto B_{S+T}$.
Since $A_*$ is generated in degrees $T$, for each $s\in M \setminus T$,
$f_s$ is determined by $f_T=\bigoplus_{t\in T} f_t$.  
In particular, $f_*=1$ if, and only if, $f_T=1$.
Hence, the mapping 
\[
f_*\mapsto 
\left(\bigoplus_{t\in T} f_t,\bigoplus_{s\in S}f_s,\bigoplus_{s\in S,t\in T}f_{s+t}\right) 
\]
is injective.
The algorithm constructs the inverse image of this injection, and
so returns the coset of graded isomorphisms $A_*\to B_*$.  
\end{proof}

\begin{proof}[Proof of Theorem~$\ref{thm:main}$]
We analyze the complexity of Algorithm~\ref{algo:simple-iso-graded}.
If the algorithm discovers that $A_*\not\cong B_*$,
then it terminates. Hence, the case that dominates complexity is $A_*\cong B_*$. 
Using Algorithm~\ref{algo:isotopism-coset}, 
$I={\rm Iso}(U_*,V_*)$ is constructed in time
\[
O(\min\{|\Aut(A_T)|,|\Aut(A_S),|\Aut(A_{S+T})|\}(\dim A_*)^{2\omega}\log^2 |K|),
\]
and hence in time
$O(|\Aut(A_{S+T})|(\dim A_*)^{2\omega}\log^2 |K|)$.
The remaining time to construct ${\rm Iso}(A_*,B_*)$  is $O(|I|)$.
Observe that 
\begin{align*}
	|I| & = |\Aut(U_*)|
		\leq |\mathcal{M}(U_*)^{\times}|\cdot |\Aut(A_{S+T})|.
\end{align*}
The complexity stated in Theorem~\ref{thm:main} now follows by substituting $T=\{1\}$ and $S=\{s\}$, 
where $s$ is the largest positive integer such that $A_{s+1}\neq 0$.
\end{proof}

%%%%%%%%%
\subsection{Selecting optimal indices}
\label{sec:selection}
We now discuss the issue left open in 
Line $1$ of Algorithm~\ref{algo:simple-iso-graded}: how to choose the subset, $S$, 
of optimal indices.  Our aim is to predict the order of 
\[
{\rm Aut}\left(A_T\times A_S\bmto A_{S+T}\right)
\] 
without computing it.  This allows us in Line $1$ of 
Algorithm~\ref{algo:simple-iso-graded} 
to sample several subsets $S$ to find a selection whose estimated work is 
either minimal,
or below an acceptable threshold.
If $U_*$ is the bimap $A_T\times A_S\bmto A_{S+T}$, then
by definition there is an exact sequence
\begin{align}\label{eq:adj-w-ext}
	1 \to \mathcal{M}(U_*)^{\times} \to {\rm Aut}(U_*) \to {\rm Aut}(U_0),
\end{align}
from which we immediately obtain the bound  
\begin{equation}
\label{eq:basic-bound}
|{\rm Aut}(U_*)|\leq |\mathcal{M}(U_*)^{\times}|\cdot q^{(\dim U_0)^2}.
\end{equation}

This bound suffices to prove our main theorems, 
but more precise bounds on $|\Aut(U_*)|$ can be obtained with additional work. 
We include a brief discussion here both because our implementation uses the 
better bounds,
and also because future analyses of the complexity of our algorithm for
specific families of inputs may require them.

In \cite{BW:autotopism}*{Theorem~3.2} 
a property of autotopisms is given which leads to a general bound on 
$|{\rm Aut}(U_*)|$. However, a better
bound using the rings defined in Section~\ref{sec:bimap-op}
may be derived from the exact sequences in
\cite{Wilson:Skolem-Noether}*{Theorem~1.2}.
Let ${\rm Out}(R)$ be the group of outer automorphisms of a
ring $R$. 
If $R$ is equipped with an involution $a\mapsto \bar{a}$, 
then $R^{\#}=\{a\in R\colon a\bar{a}=1\}$ denotes
its group of unitary elements, and
${\rm Out}^{\#}(R)$ is the subgroup of ${\rm Out}(R)$
commuting with the involution.

%%%%%
\begin{prop}
\label{prop:first-count}
For a bimap $U_*=(U_2, U_1, U_0,\circ)$ the following holds:
\begin{align*}
	\frac{|\mathcal{T}(U_*)^{\times}|}{|\mathcal{C}(U_*)^{\times}|}
		& \leq |{\rm Aut}(U_*)| \\
		& \leq
		|\mathcal{T}(U_*)^{\times}|\cdot |{\rm Out}(\mathcal{T}(U_*))| \cdot
		\min\{|{\rm Aut}_{\mathcal{T}(U_*)}(U_i)| : 0\leq i\leq 2\}.
\end{align*}
If $U_*$ is weakly Hermitian then $\mathcal{T}(U_*)$ and $\mathcal{C}(U_*)$
are rings with involutions, and
\begin{align*}
	\frac{|\mathcal{T}(U_*)^{\#}|}{|\mathcal{C}(U_*)^{\#}|}
		& \leq |\Psi{\rm Isom}(U_*)| \\
		& \leq
		|\mathcal{T}(U_*)^{\#}|\cdot |{\rm Out}^{\#}(\mathcal{T}(U_*))| \cdot
		\min\{|{\rm Aut}_{\mathcal{T}(U_*)}(U_i)| : 0\leq i\leq 2\}.
\end{align*}
Each bound can be computed in time polynomial in $\sum_{i=0}^2 \dim U_i$.
\end{prop}

%%%%%%%%%
%%%%%%%%%
\section{Proof of Theorem~\ref{thm:Lie-main}}
\label{sec:nilpotent-Lie}
The efficiency of our test for isomorphism between graded algebras $A_*$ and $B_*$ depends 
critically on two conditions: first, we can find a homogeneous 
component $A_{S+T}$ of moderate size; secondly, the order of 
${\rm Aut}(A_T\times A_S\rightarrowtail A_{S+T})$ is manageable. 
In this section, we consider a natural family of nilpotent Lie algebras whose
basic parameters illustrate the performance of our isomorphism test. 
In doing so we prove Theorem~\ref{thm:Lie-main}.
\medskip

Every nilpotent matrix Lie algebra $\mathfrak{L}_*\leq \mathfrak{gl}(V)$ has a nontrivial
0-eigenspace $V_1$.  
Recursively, for $i\geq 1$, let $V_{i+1}\leq V$ so that $V_{i+1}/V_i$ is the 0-eigenspace 
of the representation of ${\frak L}_*$ on $V/V_i$. The flag $0<V_1<\ldots<V_{\ell}=V$ is  
denoted $\mathcal{F}({\frak L}_*)$.
Conversely, associated to each flag $\mathcal{F}$ of $V$
is a unique maximal nilpotent Lie subalgebra ${\frak P}(\mathcal{F})_*\leq \mathfrak{gl}(V)$ 
such that 
$\mathcal{F}({\frak P}(\mathcal{F})_*)=\mathcal{F}$.  
We say ${\frak L}_*$ is {\em dense} if 
\[
[{\frak L}_*,{\frak L}_*]=[{\frak P}(\mathcal{F}({\frak L}_*))_*,{\frak P}(\mathcal{F}({\frak L}_*))_*].
\]

%%%%%%%%%
\subsection{Comparing results}
Before presenting the proof of Theorem~\ref{thm:Lie-main},
we pause to compare its complexity to that of other
algorithms for algebra isomorphism, and to other well known 
computational problems.
It is helpful to use {\em L-notation}
which we define in terms of logarithms to base $q=|K|$:
\begin{align*}
	L_n[\alpha,c] & = q^{(c+o(1))(\log n)^{\alpha}(\log\log n)^{\delta(0,\alpha)}},
	& \mbox{where}~\delta(0,\alpha) = \left\{\begin{array}{cc} 0 & \alpha=0\\ 1 & \alpha>0\end{array}\right..
\end{align*}
This function interpolates between polylogarithms $L_n[0,c]\in \tilde{O}((\log n)^c)$, 
polynomials $L_n[1,c]\in \tilde{O}(n^c)$, quasi-polynomials 
$L_n[2,c]\in \tilde{O}(n^{c\ln n})$, and so forth.  For context, the heuristic estimates (in base $e$) 
for the cost to factor an integer $n$ are
$L_n[\frac{1}{2},1]$ for the quadratic sieve, and $L_n[\frac{1}{3},\sqrt[3]{64/9}]$ for the number field sieve.  
\medskip

Assuming the size of the field is constant,
the complexity of graded algebra isomorphism in Theorem~\ref{thm:main} is
$L_n[2,\frac{1}{2}]$.
\medskip

We compare this to the general isomorphism test for Lie algebras presented in~\cite{Eick:Lie}.
That algorithm must select the first 
homogeneous component ${\frak L}_1$.  In the worst case---where 
${\frak L}_*={\frak L}_1\oplus {\frak L}_2$ and $\dim {\frak L}_1=r\dim {\frak L}_*$ for
some constant $r\leq 1$---it exhaustively searches ${\rm Aut}({\frak L}_1)$, 
resulting in a complexity of $L_n[2,1]$.  For cases in which $\dim {\frak L}_1=r \sqrt{\dim {\frak L}_*}$, however,
the complexity improves to $L_n[1,c]$. This shows how the Hilbert series $\sum_i (\dim {\frak L}_i) x^i$ of the input 
influences the complexity of~\cite{Eick:Lie}; as we explain below, the Hilbert series 
exerts an influence over the 
complexity of our algorithm that is both more
subtle and more emphatic. 
\medskip

Another recent approach to isomorphism~\citelist{\cite{BW:autotopism}\cite{Wilson:Skolem-Noether}} 
can be applied to graded Lie algebras 
of class $2$.  This 
exploits invariant algebras of bimaps and is particularly effective when
there are large automorphism groups.  This method also has 
complexity ranging from a worst case of $L_n[2,1]$ down to nearly optimal run times of 
$L_n[0,2\omega]$ for inputs such as generalized Heisenberg Lie algebras. 
\medskip

Now, let us consider the complexity of our current algorithm as it appears in Theorem~\ref{thm:Lie-main}.  
First, $\varepsilon$ measures the ``area" occupied by the dense algebra ${\frak L}_*$.
The presence of a few large blocks around the middle of 
$\mathcal{F}({\frak L}_*)$ makes 
$\varepsilon$ comparatively small, 
and hence slows down the performance of our algorithm.   
Secondly, $m$ is the dimension of the smallest ${\frak L}_{s+1}$ we encounter, 
subject to the product
${\frak L}_1\times{\frak L}_s\bmto {\frak L}_{s+1}$ being nondegenerate
(the condition $s\leq 1 + d/2$ in the formula for $m$ ensures this). 
% Graphically, 
As illustrated in Figure~\ref{fig:matrix-select},
small values of $m^2/\varepsilon$ correspond to
super block diagonal layers near the middle that are as thin as possible. 
The dimensions of the possible layers are determined by the 
Hilbert series of the input.

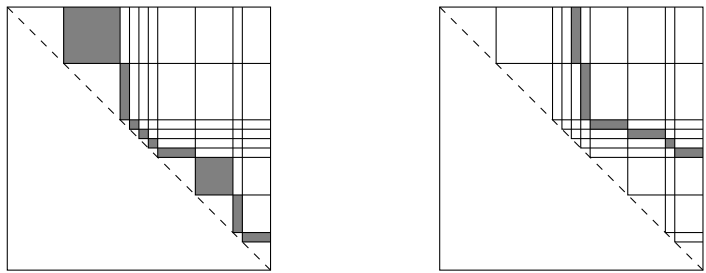
\begin{figure}[!htbp]
\begin{center}
\begin{tikzpicture}[scale=0.5]
	\fill[black!50] (1.5,7.0) -- (1.5,5.5) -- (3.0,5.5) -- (3.0,7.0) -- cycle;
	\fill[black!50] (3.0,5.5) -- (3.0,4.0) -- (3.25,4.0) -- (3.25,5.5) -- cycle;
	\fill[black!50] (3.25,4.0) -- (3.25,3.75) -- (3.5,3.75) -- (3.5,4.0) -- cycle;
	\fill[black!50] (3.5,3.75) -- (3.5,3.5) -- (3.75,3.5) -- (3.75,3.75) -- cycle;
	\fill[black!50] (3.75,3.5) -- (3.75,3.25) -- (4.0,3.25) -- (4.0,3.5) -- cycle;
	\fill[black!50] (4.0,3.25) -- (4.0,3.0) -- (5.0,3.0) -- (5.0,3.25) -- cycle;
	\fill[black!50] (5.0,3.0) -- (5.0,2.0) -- (6.0,2.0) -- (6.0, 3.0) -- cycle;
	\fill[black!50] (6.0,2.0) -- (6.0,1.0) -- (6.25,1.0) -- (6.25,2.0) -- cycle;
	\fill[black!50] (6.25,1.0) -- (6.25,0.75) -- (7.0,0.75) -- (7.0,1.0) -- cycle;	
	
	\draw (1.5,7.0) -- (1.5,5.5) -- (7.0,5.5);
	\draw (3.0,7.0) -- (3.0,4.0) -- (7.0,4.0);
	\draw (3.25,7.0) -- (3.25,3.75) -- (7.0,3.75);
	\draw (3.5,7.0) -- (3.5,3.5) -- (7.0,3.5);
	\draw (3.75,7.0) -- (3.75,3.25) -- (7.0,3.25);
	\draw (4.0,7.0) -- (4.0,3.0) -- (7.0,3.0);
	\draw (5.0,7.0) -- (5.0,2.0) -- (7.0,2.0);
	\draw (6.0,7.0) -- (6.0,1.0) -- (7.0,1.0);
	\draw (6.25,7.0) -- (6.25,0.75) -- (7.0,0.75);
	
	\draw (0,0) -- (7,0) -- (7,7) --  (0,7) -- cycle;
	
	\draw[dashed] (0,7) -- (7,0);
	
\end{tikzpicture}
\hspace{2cm}
\begin{tikzpicture}[scale=0.5]
	
	\fill[black!50] (3.5,7.0) -- (3.5,5.5) -- (3.75,5.5) -- (3.75,7.0) -- cycle;
	\fill[black!50] (3.75,5.5) -- (3.75,4.0) -- (4.0,4.0) -- (4.0,5.5) -- cycle;
	\fill[black!50] (4.0,4.0) -- (4.0,3.75) -- (5.0,3.75) -- (5.0,4.0) -- cycle;
	\fill[black!50] (5.0,3.75) -- (5.0,3.5) -- (6.0,3.5) -- (6.0,3.75) -- cycle;
	\fill[black!50] (6.0,3.5) -- (6.0,3.25) -- (6.25,3.25) -- (6.25,3.5) -- cycle;
	\fill[black!50] (6.25,3.25) -- (6.25,3.0) -- (7.0,3.0) -- (7.0,3.25) -- cycle;
	
	\draw (1.5,7.0) -- (1.5,5.5) -- (7.0,5.5);
	\draw (3.0,7.0) -- (3.0,4.0) -- (7.0,4.0);
	\draw (3.25,7.0) -- (3.25,3.75) -- (7.0,3.75);
	\draw (3.5,7.0) -- (3.5,3.5) -- (7.0,3.5);
	\draw (3.75,7.0) -- (3.75,3.25) -- (7.0,3.25);
	\draw (4.0,7.0) -- (4.0,3.0) -- (7.0,3.0);
	\draw (5.0,7.0) -- (5.0,2.0) -- (7.0,2.0);
	\draw (6.0,7.0) -- (6.0,1.0) -- (7.0,1.0);
	\draw (6.25,7.0) -- (6.25,0.75) -- (7.0,0.75);
	
	\draw (0,0) -- (7,0) -- (7,7) --  (0,7) -- cycle;
	
	\draw[dashed] (0,7) -- (7,0);
	
\end{tikzpicture}
\end{center}
\caption{
Contrasting the work to construct automorphisms}
\label{fig:matrix-select}
%\caption{A judicious layer selection slithers between the larger blocks. 
%Layers that include large blocks can result in running times
%$L[2,c_1]$, whereas thin layers improve to $L[1,c_2]$.}\label{fig:matrix-select}
\end{figure}

Observe that the work to construct automorphisms by lifting from a fixed shaded layer
is roughly $q^{a^2}$, where $q$ is the size of the field and
$a$ is the area of the shaded region 
(the dimension of the corresponding homogeneous component).  
The shaded area on the left of Figure~\ref{fig:matrix-select} is roughly 3 times
that 
% of the shaded area 
on the right.  Thus, if $w$ is the work 
needed to search for automorphisms using the shaded region on the left, 
then using that 
% the shaded region 
on the right decreases the work to $w^{1/9}$.

%%%
\subsection{Multiplication tables}
We work with ${\frak L}_*$ as matrices relative to
a basis exhibiting $\mathcal{F}({\frak L}_*)$, so elements of
${\frak L}_*$ are block upper-triangular matrices with $0$'s on the diagonal.  
Although the specific entries in each block determine the structure 
of the algebra, much can be deduced simply by studying the block structure.  
For $x\in {\frak L}_*$, denote by $x_{st}$ the block of $x$ in (block) row
$s$ and (block) column $t$.  For $x,y\in {\frak L}_*$, the following is the general
formula for the product on ${\frak L_*}$:
\begin{align}
	\left[\sum_{1\leq s<t\leq d} x_{st},\sum_{1\leq u<v\leq d} y_{uv}\right]
		& = \sum_{1\leq s<u<t\leq d} x_{su} y_{ut} - y_{su}x_{ut}.
\end{align}
Each product $x_{su} y_{ut} - y_{su}x_{ut}$ may be specified by structure
constants depending solely on $(s,t)$; we demonstrate this below for $d=6$.
(As  ${\frak L}_*$ is generated 
in degree $1$, we 
display only the structure constants for each product
${\frak L}_1\times {\frak L}_n\bmto {\frak L}_{n+1}$ for $n=1,\ldots,5$.)
\smallskip

\begin{align*}
\tiny{
\begin{array}{c|ccccc|cccc|ccc|cc|c|}
& x_{12} & x_{23} & x_{34} & x_{45} & x_{56} & x_{13} & x_{24} & x_{35} & x_{46}
	& x_{14} & x_{25} & x_{36} & x_{15} & x_{26} & x_{16}\\[1pt]
\hline
x_{12} &. & a_1  & . & . & . & . & b_1 & . & . & . & c_1 & . & . & d_1 & .\\
x_{23} & -a_1^{\dagger} & .  & a_2 & . & . & . & . & b_2 & . & . & . & c_2 & . & . & .\\
x_{34} & . & -a_2^{\dagger}  & . & a_3 & . & -b_1^{\dagger} & . & . & b_3 & . & . & . & . & . & .\\
x_{45} &. & .  & -a_3^{\dagger} & . & a_4 & . & -b_2^{\dagger} & . & . & -c_1^{\dagger} & . & . & . & . & .\\
x_{56} & . & .  & . & -a_4^{\dagger} & . & . & . & -b_3^{\dagger}  & . & . & c_2^{\dagger} & . & -d_1^{\dagger} & . & .\\[1pt]
\hline
\end{array}
}
\end{align*}
\smallskip

This example illustrates two competing tensions that determine our success.
On one hand,  we want to select $s$ such that 
${\rm Aut}({\frak L}_1\times {\frak L}_s\bmto {\frak L}_{1+s})$ is small.  This 
occurs when the corresponding product has the most constraining equations, which 
happens generically when $s$ is small.  On the other hand, our method to compute 
${\rm Aut}({\frak L}_1\times {\frak L}_s\bmto {\frak L}_{1+s})$
requires that we list $\Aut(\frak{L}_{1+s})$, and the order of this group tends to decrease 
as $s$ increases.  
This explains why the best choice of $s$ is typically near the middle.

%%%%%
\subsection{Estimating \boldmath$\Aut({\frak L}_1\times {\frak L}_s\bmto {\frak L}_{1+s})$}
\label{sec:auto-est}
To prove Theorem~\ref{thm:Lie-main} it suffices to consider the bound on 
$\Aut({\frak L}_1\times {\frak L}_s\bmto {\frak L}_{1+s})$
arising from the exact sequence \eqref{eq:adj-w-ext}.  In particular we show that if 
$s\leq 1+d/2$, then  the choice $U_*=({\frak L}_1,{\frak L}_s,{\frak L}_{1+s},[,])$ of homogeneous 
component in Theorem~\ref{thm:main} leads to the complexity stated in Theorem~\ref{thm:Lie-main}.

Put $U_2:={\frak L}_1,\,U_1:={\frak L}_s$, and $U_0:={\frak L}_{1+s}$. 
It suffices to establish the bounds
\begin{align}\label{eq:easy-bound}
	|\mathcal{M}(U_*)^{\times}|
		& \leq q^{d^2} \leq |{\frak L}|^{1/\varepsilon}, &
	|{\rm Aut}(U_0)| & \leq q^{m^2d^2}\leq |{\frak L}|^{m^2/\varepsilon}.
\end{align}
As suggested in Section~\ref{sec:selection} we can derive more subtle bounds by applying 
Proposition~\ref{prop:first-count}, but the analysis offers little insight into the complexity 
of our algorithm.

We first make an observation. If $X_*^i=(X_2,X_1,X_0^i,\circ_i)$ $(i=1,2)$ are two bimaps on a common
domain $X_2\times X_1$, but with possibly different codomains and products, then we
can define a new bimap $Y_*=X_*^1\cap X_*^2$ as follows. Put $Y_2:=X_2$, $Y_1:=X_1$, 
$Y_0:=X_0^1\oplus X_0^2$, and define $\circ\colon Y_2\times Y_1\bmto Y_0$
by 
\begin{align*}
	y_2\circ y_1 := (y_2\circ_1 y_1)\oplus (y_2\circ_2 y_1).
\end{align*}
Note that $\mathcal{M}(X_*^1\cap X_*^2)=\mathcal{M}(X_*^1)\cap\mathcal{M}(X_*^2)$, 
so selecting a basis for $U_{0}$ decomposes $U_*$ as $\bigcap_i U_*^i$,
thereby making it simpler to compute the ring of adjoints.
\medskip

To compute $\mathcal{M}(U_*)$  we solve equations of the form $FA=AG^{\dagger}$, 
where $A$ is a matrix of structure constants of $U_*$ of the form
\begin{align}
\label{eq:main-panel}
A=\tiny{
	\begin{bmatrix}
	. & a_1 & \\
	 & \ddots & \ddots \\
	 & & . & a_{\ell}\\
	 -a^{\dagger}_1 & . \\
	  & \ddots & \ddots \\
	  & & -a^{\dagger}_{\ell} & .
	\end{bmatrix}.
}
\end{align}
We caution that our illustration shows the structure 
constants for a typical configuration but changing the numbers of blocks
both changes the number of boxes and alters their positions.
To solve this system of equations, we decompose $U_*$ into terms $U_*^i$ using
the block structure, as follows:
\begin{align*}
\tiny{
	\begin{bmatrix}
	. & a_1 & \\
	 & \ddots & \ddots \\
	 & & . & a_{\ell}\\
	 -a^{\dagger}_1 & . \\
	  & \ddots & \ddots \\
	  & & -a^{\dagger}_{\ell} & .
	\end{bmatrix}
	=
	\begin{bmatrix}
	. & a_1 & \\
	 & \ddots & \ddots \\
	 & & . & .\\
	 -a^{\dagger}_1 & . \\
	  & \ddots & \ddots \\
	  & & . & .
	\end{bmatrix}
	\cap\cdots \cap
		\begin{bmatrix}
	. & . & \\
	 & \ddots & \ddots \\
	 & & . & a_{\ell}\\
	 . & . \\
	  & \ddots & \ddots \\
	  & & -a^{\dagger}_{\ell} & .
	\end{bmatrix}
}
.
\end{align*}
This is not precisely a decomposition into the blocks $U_*^i$ because
each term on the right hand side is padded with zeros, thereby
defining a degenerate extension $\hat{U}_*^i$ of the desired $U_*^i$.
However, the relationship between $\mathcal{M}(\hat{U}_*^i)$ and 
$\mathcal{M}(U_i^*)$ is straightforward:
\begin{align*}
	\mathcal{M}(\hat{U}^i_*) & = \mathcal{M}(U^i_*)
		+\Hom\left(U^i_2,(U^i_1)^{\perp}\right)\times 
		\Hom\left(U^i_1,(U^i_2)^{\perp}\right),~~\mbox{and} \\
\mathcal{M}\left(\tiny{\begin{bmatrix}
	. & . & \\
	 & . & a_i \\
	 & & . & .\\
	 . & . \\
	  & -a_i^{\dagger} & . \\
	  & & . & .
	\end{bmatrix}}\right)	
& =\left\{
	\tiny{
	\left(
	\begin{bmatrix}
	* & . & * & * & . & *\\
	. & \alpha & . & . & \beta & . \\
	* & . & * & * &. & *\\
	* & . & * & * & . & * \\
	. & \gamma & . & . & \delta & . \\
	* & . & * & * & . & * \\
	\end{bmatrix},
	\begin{bmatrix}
	* & . & * & . & *\\
	. & \delta^{\dagger} & . & -\beta^{\dagger} & . \\
	* & . & * & . & *\\
	. & -\gamma^{\dagger} & . & \alpha^{\dagger} & . \\
	* & . & * & . & *\\
	\end{bmatrix}\right)	
	}
	\right\}
\end{align*}
Crucially, the positions of the zeros in the matrices of the
$\mathcal{M}(U_*^i)$ shift as $i$ changes.  
Recalling that $\mathcal{M}(U_*)=\bigcap_i
\mathcal{M}(U_*^i)$ we discover that
\begin{align*}
\mathcal{M}\left(\tiny{\begin{bmatrix}
	. & a_1 & \\
	 & \ddots & \ddots \\
	 & & . & a_{\ell}\\
	 -a^{\dagger}_1 & . \\
	  & \ddots & \ddots \\
	  & & -a^{\dagger}_{\ell} & .
	\end{bmatrix}
}\right)	
& \cong \left\{
	\tiny{
	\begin{bmatrix}
	\alpha_0^{\phantom{\dagger}} &  &  &   \\
	\beta_1 & \alpha_1 &  &   \\
	 &  & \ddots &   &   \\
	 &  &  &  &  \\
	 &   &  &  &  \\
	 &  &    & \alpha_{\ell-1} & \beta_2 \\
	 &  &    &  & \alpha_{\ell} \\
	\end{bmatrix}}
	\right\}.
\end{align*}
Here, $\alpha_i$
is an $(e_i\times e_i)$-matrix where $e_i=\dim U_1^i$, $\beta_1$ is
a $(e_1\times e_2)$-matrix, and $\beta_{2}$ is a $(e_{\ell-1}\times e_{\ell})$-matrix.
Hence,
\begin{align}
	\dim \mathcal{M}(U_*) & \leq e_1\cdot e_2+e_{\ell-1}\cdot e_{\ell}
		+\sum_{i=1}^{\ell} e_i^2.
\end{align}
We remark that the specific entries $\alpha_i$ are further constrained, since they
are adjoints of the individual bimaps $U_*^i$.
Finally, we observe that $e_i=\dim V_i-\dim V_{i-1}$, where $0=V_0<V_1<\cdots<V_{\ell}=V$
is the fixed point flag for ${\frak L}_*$, from which the desired bound
$|\mathcal{M}(U_*)^{\times}|\leq |{\frak L}_*|^{1/\varepsilon}$  in~\eqref{eq:easy-bound} follows.
The bound on ${\rm Aut}(U_0)$ is straight-forward, since 
\begin{align*}
	\dim U_0=\dim {\frak L}_{1+s} = \sum_{i=1}^{s} d_i d_{i+s+1}.
\end{align*}
Thus, $|{\rm Aut}(U_*)|\leq |\mathfrak{L}_*|^{1/\varepsilon}|\cdot |\mathfrak{L}_*|^{m^2/\varepsilon}$,  so
Theorem~\ref{thm:Lie-main} now follows from Theorem~\ref{thm:main}.

%%%%%
\subsection{Number of isomorphism types}
% Finally we demonstrate the variability amongst the 
%We now consider the number of isomorphism types of
%the Lie algebras we have considered.  An immediate bound is provided by the number
%of partitions, which grows exponentially in the dimension.  Having fixed the
%partition, the number of isomorphism types among dense subalgebras 
%is also large. 
% offers considerably more variability in isomorphism type.  

An immediate bound to the number of isomorphism types 
in $\mathcal{F}_q(d_1,\dots,d_{\ell})$  is 
provided by the number
of partitions, which grows exponentially in the dimension.  Having fixed the
partition, the number of isomorphism types among dense subalgebras 
is also large. 
To see this, observe that blocks yield numerous characteristic subalgebras.  
We can, for example, remove rows and columns resulting
in characteristic quotients.  If ${\frak P}(\mathcal{F})$ is symmetric 
along the anti-diagonal, then we must remove both rows and columns;
otherwise we can remove these independently.
By removing sufficiently many rows/columns, we may assume that
the partition has three parts. 
The number of isomorphism types for algebras based on such partitions
may be estimated using arguments similar to those
of Higman and Sims~\cite{BNV:enum}*{Chapter 2}
and suffice for the bound stated in Theorem~\ref{thm:Lie-main}.

%%%%%
\section{Pseudo-isometries}
\label{sec:pseudo}
The occasional presence of symmetry provides an opportunity to
improve the complexity of our isomorphism test for graded algebras.
For example, if
$A_{*}$ is a commutative algebra
generated in degrees $T$, then the map
$A_T\times A_T\bmto A_{T+T}$ is symmetric;
if $A$ is a Lie algebra, then $A_T\times A_T\bmto A_{T+T}$ is alternating. 
In such cases, the group of graded automorphisms of $A_{*}$ embeds
in $ \pseudo(A_T\times A_T\bmto A_{T+T}) $,
the group of pseudo-isometries
defined in Section~\ref{sec:weakly};  
typically this is a proper subgroup of 
$\Aut(A_T\times A_T\bmto A_{T+T})$.
To take advantage of this observation, we must lift 
elements of $\Aut(A_{T+T})$ directly to 
$\pseudo(A_T\times A_T\bmto A_{T+T})$ 
 rather than to 
$\Aut(A_T\times A_T\bmto A_{T+T})$. 
Ivanyos and Qiao~\cite{IQ} devised an algorithm to do this.
\smallskip

Our objective is to solve the following problem.
%\medskip

\begin{problem}
{\sc PseudoIsometryCoset}
\begin{description}
\item[Given]  weakly Hermitian bimaps $U_*$ and $V_*$.

\item[Return] the coset of all pseudo-isometries of $f_*\colon U_*\to V_*$.
\end{description}
\end{problem}
%\medskip

Once again, we first consider the restricted version of the problem
that arises by insisting that maps induce the identity on the codomain. The analogue
of a principal isotopism in this weakly Hermitian context is an {\em isometry}.

\begin{defn}
Let $U_*$ and $V_*$ be weakly Hermitian bimaps. If 
$h\colon U_0\to V_0$ is a linear isomorphism, then $U_*$ and $V_*$
are {\em $h$-isometric} if there is a pseudo-isometry $f_*\colon U_* \to V_*$ 
with $f_0=h$. If $U_*= V_*$ and $h = 1$, then $\Isom(U_*)$ is the {\em isometry group} of $U_*$.
\end{defn}

Consider the following problem and its coset analogue.
%\medskip

\begin{problem}
{\sc IsometryGroup}
\begin{description}
\item[Given] a weakly Hermitian bimap $U_*$.

\item[Return] generators for $\Isom(U_*)$.
\end{description}
\end{problem}

\begin{problem}
{\sc IsometryCoset}
\begin{description}
\item[Given]  weakly Hermitian bimaps $U_*, V_*$ and a linear isomorphism
$h\colon U_0 \to V_0$.

\item[Return] the coset of all isometries from $U_*\to V_*$ extending $h$.
\end{description}
\end{problem}
%\medskip

Polynomial-time solutions to {\sc IsometryGroup} and {\sc IsometryCoset}
when $K$ has odd characteristic 
appear in \cite[Theorem~1.2]{BW:isom} and \cite{IQ}, respectively.
We combine 
these and an adaption of Algorithm~\ref{algo:isotopism-coset} to solve {\sc PseudoIsometryCoset}.

%%%%%
\section{Heuristics}
\label{sec:local}
The family of graded Lie algebras described in Section~\ref{sec:nilpotent-Lie}
illustrates decisively the importance of the overarching ``layer selection" 
philosophy of our approach.  In this section, we assume that we have 
chosen a bimap $U_*=(U_2,U_1,U_0,\circ)$,
and consider the problem of constructing its group of autotopisms.

In Line~4 of Algorithm~\ref{algo:isotopism-coset}, we have a 
bimap $U_*=(U_2,U_1,U_0,\circ)$ and we must choose a group 
$G$ with $\Aut(U_*)|_{U_0} \leq G \leq \Aut(U_0)$. 
Without further information, we are forced to
choose $G=\Aut(U_0)$, a group which is often too large to search exhaustively. 
In this section we introduce heuristics to cut down the order of $G$.

One idea is to compose $U_2\times U_1\bmto U_0$ with 
projections $\pi\colon U_0\to K^c$ and record isotopism invariants of the
resulting bimap $U_*^{\pi}\colon U_2\times U_1 \bmto K^c$.  
Thus, we label subspaces of the dual space $U_0^{\dagger}=\Hom_K(U_0,K)$.
There are $O(|K|^{c( \dim U_0)})$ subspaces of dimension $c$, 
by contrast to the $O(|K|^{(\dim U_0)^2})$ elements of $\Aut(U_0)$ we would
otherwise be forced to list. It is therefore reasonable to list and label 
the former for small values of $c$. One then chooses $G$ as the subgroup of $\Aut(U_0)$
that preserves labels. 

In practice, 
we treat the projective geometry of $U_0^{\dagger}$ as a complete, 
colored graph, where vertices and edges are colored according to isotopism
invariants. We then construct $G$ as the automorphism group of the colored graph.
Since no polynomial-time algorithm is known to construct the automorphism group 
of a graph, we cannot satisfactorily bound the complexity of this task.
In practice, the construction of the graph is the more expensive task since 
the NAUTY algorithm~\cite{nauty} is extremely fast on generic graphs.
Of course, there is no guarantee that $G$ is a {\it proper} subgroup of $\Aut(U_0)$. 
We discuss this matter further in Section~\ref{sec:imp}.

We now describe our labels for 
the points (vertices) and lines (edges) of $U_0^{\dagger}$.
We also examine another heuristic technique
called {\em fingerprinting} that was first introduced in~\cite{ELGOB}.

%%%%%%
\subsection{Vertex labels}
For each epimorphism $\pi\colon U_0\to K$, we
define a $K$-bilinear form $U_*^{\pi}$.
In particular, if we fix bases for $U_2\cong K^{a\times 1}$ and 
$U_1\cong K^{1\times b}$, then 
$U_*^{\pi}\colon K^a\times K^b\bmto K$ is represented by its
Gram matrix $D\in K^{a\times b}$ defined by
\begin{align*}
	(e_i* e_j)\pi & = e_i D e_j.
\end{align*}
Base changes to $U_2$ and $U_1$ leave the rank of the 
Gram matrix unchanged. We use the rank of $D$ to label 
$\langle \pi\rangle\in \mathbb{P}(U_0^{\dagger})$; this can be computed using \cite{Wilson:GramSchmidt}.

%%% 
\subsubsection{Examples.}
\label{sec:examples}
To illustrate subtleties arising from vertex labels, 
we consider two alternating
bimaps $K^4\times K^4\bmto K^3$ specified 
by systems of alternating forms.
Let
\begin{align*}
D & = 
\left[
\begin{bmatrix}
0 & 1 & 0 & 0\\
-1 & 0 & 0 & 0\\
0 & 0 & 0 & 1\\
0 & 0 & -1 & 0 
\end{bmatrix},
\begin{bmatrix}
0 & 0 & 0 & 0\\
0 & 0 & 1 & 0\\
0 & -1 & 0 & 0\\
0 & 0 & 0 & 0 
\end{bmatrix},
\begin{bmatrix}
0 & 0 & 0 & 0\\
0 & 0 & 0 & 1\\
0 & 0 & 0 & 0\\
0 & -1 & 0 & 0 
\end{bmatrix}\right],
\end{align*} 
\begin{align*}
	E & = 
\left[
\begin{bmatrix}
0 & 1 & 0 & 0\\
-1 & 0 & 0 & 0\\
0 & 0 & 0 & 1\\
0 & 0 & -1 & 0 
\end{bmatrix},
\begin{bmatrix}
0 & 0 & 0 & 0\\
0 & 0 & 1 & 0\\
0 & -1 & 0 & 0\\
0 & 0 & 0 & 0 
\end{bmatrix},
\begin{bmatrix}
0 & 0 & 0 & 1\\
0 & 0 & 0 & 0\\
0 & 0 & 0 & 0\\
 -1 & 0 & 0 & 0 
\end{bmatrix}\right].
\end{align*}
In each case the dual space $\Hom(K^3,K)$ determines a projective plane. 
For each point $P=(x_1:x_2:x_3)$ in the plane
we obtain matrices
\begin{align*}
	M_D(P) & \equiv x_1 D[1]+x_2 D[2]+x_3 D[3] 
		\equiv \begin{bmatrix}
		0 & x_1 & 0 & 0\\
		-x_1 & 0 & x_2 & x_3\\
		0 & -x_2 & 0 & x_1\\
		0 & -x_3 & -x_1 & 0 	
	\end{bmatrix}
	\pmod{K^{\times}},
\end{align*}
\begin{align*}
	M_E(P) & \equiv x_1 E[1]+x_2 E[2]+x_3 E[3] 
	\equiv \begin{bmatrix}
		0 & x_1 & 0 & x_3\\
		-x_1 & 0 & x_2 & 0\\
		0 & -x_2 & 0 & x_1\\
		-x_3 & 0 & -x_1 & 0 	
	\end{bmatrix}
	\pmod{K^{\times}}.
\end{align*}
As indicated, these matrices are unique up to nonzero scalars, 
and scalars do not affect their ranks.
Since $M_D(P)$ and $M_E(P)$ are nonzero and alternating, the only possible 
ranks are $2$ and $4$, distinguished by whether the determinants
\begin{align*}
\det(M_D(P))  = x_1^4, &  & \det(M_E(P))  = x_1^4+2x_1^2 x_2x_3+x_2^2 x_3^2
\end{align*}
are zero or nonzero. Hence,
the points of rank $2$ for $D$ lie on the hyperplane at infinity, 
namely $(0:0:1)+(0:1:0)$, while
all other points have rank $4$.  
For $E$, the only rank 2 points with $x_1=0$ satisfy $x_2x_3=0$, so there are just two, 
namely $(0:1:0)$ and $(0:0:1)$.
The remaining rank 2 points have the form $(1:a:-1/a)$ for $a\in K^\times$. Hence, 
there are also $|K|+1$ total
points of rank 2, but they clearly do not lie on a line. We illustrate the situation for 
the field of order $3$ in~Figure~\ref{fig:ex-1}. 

It follows that the alternating bimaps defined by the systems
$D$ and $E$ are not pseudo-isometric, a fact that will be transparent when we consider line labels.

\begin{figure}[!htbp]
\begin{center}
\begin{tikzpicture}[
mydot/.style={
  draw,
  circle,
  fill=black,
  inner sep=1.5pt}
]
\tikzstyle{vertex}=[circle,fill=black!25,minimum size=17pt,inner sep=0pt]

\node[vertex] (p100) at (0,0) {100}; 
\node[vertex] (p110) at (1,0) {110};
\node[vertex] (p120) at (2,0) {120};

\node[vertex] (p101) at (0,1) {101};
\node[vertex] (p111) at (1,1) {111};
\node[vertex] (p121) at (2,1) {121};

\node[vertex] (p102) at (0,2) {102};
\node[vertex] (p112) at (1,2) {112};
\node[vertex] (p122) at (2,2) {122};

\node[vertex] (p010) at (3.5,1) {010};
\node[vertex] (p011) at (2.75,2.75) {011};
\node[vertex] (p012) at (-0.75,2.75) {012};

\node[vertex] (p001) at (1,3.5) {001};

\draw[color = black] (p100) -- (p110) -- (p120) edge[bend right=25] (p010);
\draw[color = black] (p101) -- (p111) -- (p121) -- (p010);
\draw[color = black] (p102) -- (p112) -- (p122) edge[bend left=25] (p010);

\draw[color = black] (p100) -- (p101) -- (p102) edge[bend left=25] (p001);
\draw[color = black] (p110) -- (p111) -- (p112) -- (p001);
\draw[color = black] (p120) -- (p121) -- (p122) edge[bend right=25] (p001);

\draw[color = black] (p100) -- (p111) -- (p122) -- (p011);
\draw[color = black] (p101) -- (p112) edge[bend left=25] (p011); \draw[color =black] (p011) edge[bend left=45] (p120);
\draw[color = black] (p110) -- (p121) edge[bend right=25] (p011); \draw[color =black] (p011) edge[bend right=45] (p102);

\draw[color = black] (p012) -- (p102) -- (p111) -- (p120);
\draw[color = black] (p012) edge[bend right=45] (p100) edge[bend left=25] (p112); \draw[color=black] (p112) -- (p121);
\draw[color = black] (p012) edge[bend left=45] (p122) edge[bend right=25] (p101); \draw[color=black] (p101) -- (p110);

\draw[color= black] (p012) edge[bend left=15] (p001);
\draw[color= black] (p001) edge[bend left=15] (p011);
\draw[color= black] (p011) edge[bend left=15] (p010);

\end{tikzpicture}%
\end{center}

\begin{tikzpicture}[
mydot/.style={
  draw,
  circle,
  fill=black,
  inner sep=1.5pt}
]
\tikzstyle{vertex}=[circle,fill=black!25,minimum size=17pt,inner sep=0pt]

\node[vertex] (p100) at (0,0) {4}; 
\node[vertex] (p110) at (1,0) {4};
\node[vertex] (p120) at (2,0) {4};

\node[vertex] (p101) at (0,1) {4};
\node[vertex] (p111) at (1,1) {4};
\node[vertex] (p121) at (2,1) {4};

\node[vertex] (p102) at (0,2) {4};
\node[vertex] (p112) at (1,2) {4};
\node[vertex] (p122) at (2,2) {4};

\node[vertex] (p010) at (3.5,1) {2};
\node[vertex] (p011) at (2.75,2.75) {2};
\node[vertex] (p001) at (1,3.5) {2};
\node[vertex] (p012) at (-0.75,2.75) {2};

\draw[color = red] (p100) -- (p110) -- (p120) edge[bend right=25] (p010);
\draw[color = red] (p101) -- (p111) -- (p121) -- (p010);
\draw[color = red] (p102) -- (p112) -- (p122) edge[bend left=25] (p010);

\draw[color = red] (p100) -- (p101) -- (p102) edge[bend left=25] (p001);
\draw[color = red] (p110) -- (p111) -- (p112) -- (p001);
\draw[color = red] (p120) -- (p121) -- (p122) edge[bend right=25] (p001);

\draw[color = red] (p100) -- (p111) -- (p122) -- (p011);
\draw[color = red] (p101) -- (p112) edge[bend left=25] (p011); \draw[color =red] (p011) edge[bend left=45] (p120);
\draw[color = red] (p110) -- (p121) edge[bend right=25] (p011); \draw[color =red] (p011) edge[bend right=45] (p102);

\draw[color = red] (p012) -- (p102) -- (p111) -- (p120);
\draw[color = red] (p012) edge[bend right=45] (p100) edge[bend left=25] (p112); \draw[color=red] (p112) -- (p121);
\draw[color = red] (p012) edge[bend left=45] (p122) edge[bend right=25] (p101); \draw[color=red] (p101) -- (p110);

\draw[color= green] (p012) edge[bend left=15] (p001);
\draw[color= green] (p001) edge[bend left=15] (p011);
\draw[color= green] (p011) edge[bend left=15] (p010);

\end{tikzpicture}%
\qquad
\begin{tikzpicture}[
mydot/.style={
  draw,
  circle,
  fill=black,
  inner sep=1.5pt}
]
\tikzstyle{vertex}=[circle,fill=black!25,minimum size=17pt,inner sep=0pt]

\node[vertex] (p100) at (0,0) {4}; 
\node[vertex] (p110) at (1,0) {4};
\node[vertex] (p120) at (2,0) {4};

\node[vertex] (p101) at (0,1) {4};
\node[vertex] (p111) at (1,1) {4};
\node[vertex] (p121) at (2,1) {2};

\node[vertex] (p102) at (0,2) {4};
\node[vertex] (p112) at (1,2) {2};
\node[vertex] (p122) at (2,2) {4};

\node[vertex] (p010) at (3.5,1) {2};
\node[vertex] (p011) at (2.75,2.75) {4};
\node[vertex] (p001) at (1,3.5) {2};
\node[vertex] (p012) at (-0.75,2.75) {4};

\draw[color = blue] (p100) -- (p110) -- (p120) edge[bend right=25] (p010);
\draw[color = red] (p101) -- (p111) -- (p121) -- (p010);
\draw[color = red] (p102) -- (p112) -- (p122) edge[bend left=25] (p010);

\draw[color = blue] (p100) -- (p101) -- (p102) edge[bend left=25] (p001);
\draw[color = red] (p110) -- (p111) -- (p112) -- (p001);
\draw[color = red] (p120) -- (p121) -- (p122) edge[bend right=25] (p001);

\draw[color = green] (p100) -- (p111) -- (p122) -- (p011);
\draw[color = blue] (p101) -- (p112) edge[bend left=25] (p011); \draw[color =blue] (p011) edge[bend left=45] (p120);
\draw[color = blue] (p110) -- (p121) edge[bend right=25] (p011); \draw[color =blue] (p011) edge[bend right=45] (p102);

\draw[color = green] (p012) -- (p102) -- (p111) -- (p120);
\draw[color = red] (p012) edge[bend right=45] (p100) edge[bend left=25] (p112); \draw[color=red] (p112) -- (p121);
\draw[color = green] (p012) edge[bend left=45] (p122) edge[bend right=25] (p101); \draw[color=green] (p101) -- (p110);

\draw[color= red] (p012) edge[bend left=15] (p001);
\draw[color= red] (p001) edge[bend left=15] (p011);
\draw[color= red] (p011) edge[bend left=15] (p010);

\end{tikzpicture}%

\caption{On top is the geometry $PG(2,3)$ labeled by homogeneous coordinates. 
The labeled and colored versions corresponding to bimaps
$D$ and $E$ are shown on the left and right, respectively.}
\label{fig:ex-1}
\end{figure}
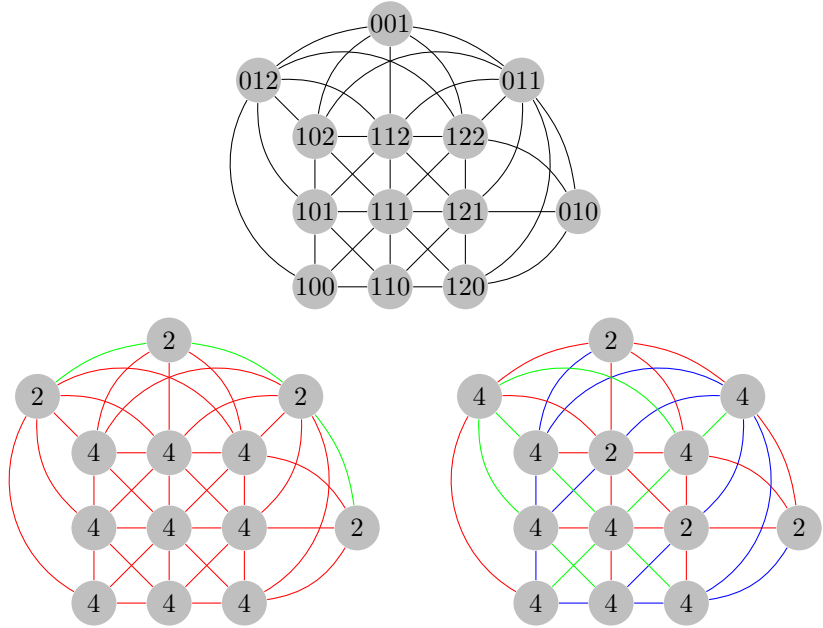

%%%%%
\subsection{Line labels}
\label{sec:line-labels}
For each epimorphism $\lambda\colon U_0\to K^2$, 
we consider the composition 
$U_*^{\lambda}\colon U_2\times U_1\bmto U_0 \to K^2$.
Choices of bases for $U_2,U_1,U_0$ determine
pairs of matrices associated to $U_*^{\lambda}$,
and the label of $\langle \lambda \rangle$ is a polynomial invariant
derived from this pair.

Characterizations of {\em indecomposable} pairs of matrices
defined over the complex numbers were given 
by Kronecker and later for arbitrary fields by 
Dieudonn\'{e}~\cite{Dieudonne}.  Using~\cite{Dieudonne},
Scharlau~\cite{Scharlau} characterized indecomposable
skew-symmetric pairs.
The reduction to indecomposable pairs is accomplished
by computing a fully-refined orthogonal decomposition 
of  $U_*^{\lambda}$ using an algorithm of Wilson~\cite{Wilson:find-central}.
Based on Scharlau's characterization,
the aforementioned polynomial invariant for alternating pairs was first introduced
by Vishnivitski\u{\i}~\cite{Vish:2}
for prime fields. It was later shown in~\cite[Theorem 3.22]{BMW} to determine alternating pairs
up to pseudo-isometry over all finite fields.
The invariant may be adapted to determine all $U_*^{\lambda}$ up
to autotopism.
\smallskip

There are two consequences. 
First, the best possible labels for the lines of our projective geometry
can be computed efficiently using the algorithm of~\cite{BMW}.
Secondly, there is a far greater variety of possible line labels than
there are point labels, and so the study of lines 
may reveal significant global constraints. 
% makes it much more likely that 
% significant global 
% constraints can be revealed by moving beyond just points.

%%%%
\subsubsection{Examples, revisited.} 
\label{sec:ex-rev}
Although the bimaps $D$ and $E$ in Section~\ref{sec:examples} are too small
to illustrate great variability, we already begin to see differences in behaviour. 
We illustrate this for the field of order $3$ in Figure~\ref{fig:ex-1}. 
For bimap $D$ there are only two line labels. The first (colored green) labels 
the single line at infinity
$(0:0:1)+(0:1:0)$, and the other lines are all labeled the same (colored red).
For the bimap $E$ there are three labels (colored blue, green, and red). 

%%%%%%%%%%%%%
\subsection{Fingerprinting}
\label{subsec:finger}
This process was introduced in~\cite{ELGOB} in the 
context of constructing automorphism groups of $p$-groups. Here, we examine 
it in more detail and generalize it to our setting.

First, consider a linear map $f\colon V\to W$. Every subspace of $V$ is mapped to 
a subspace of $W$, and
the preimage of a subspace of $W$ is a subspace of $V$. This correspondence is 
order-preserving and allows us to
compare the projective geometries of $V$ and $W$.

Next, consider a bilinear map $U_*\colon U_2\times U_1\bmto U_0$.   
We compare subspaces of $U_2\otimes U_1$ and $U_0$.  
Since most of the elements of $U_2\otimes U_1$ are not pure tensors, the preimages that 
are pure tensors provide a largely {\it ad hoc} distribution of subspaces.  This presents
an opportunity to discover properties of $U_*$ that can be used to distinguish between bimaps.

\begin{defn}
The {\em left idealizer} of $S\leq U_0$ is 
$\lambda(S)=\{u_2\in U_2 \colon u_2 \circ U_1\leq S\}$ and the {\em right 
idealizer} is $\rho(S)=\{u_1\in U_1\colon U_2\circ u_1 \leq S\}$.
These are subspaces of $U_2$ and $U_1$, respectively.
\end{defn}

We collect some basic properties of idealizers (stated just for left idealizers).

\begin{lem}
For subspaces $S$ and $T$ of $U_0$, the following hold:
\begin{enumerate}[(i)]
\item $S\leq T~\Longrightarrow~\lambda(S)\leq \lambda(T)$;
\item $\lambda(S\cap T)= \lambda(S)\cap\lambda(T)$; and
\item $\lambda(S)+\lambda(T)\leq \lambda(S+T)$.
\end{enumerate}
\end{lem}

Since the subspaces $\lambda(S)+\lambda(T)$ and $ \lambda(S+T)$
are typically distinct, there is an opportunity to
discover hidden invariants of a bimap. 
Using $\lambda$ and $\rho$, uniform incidences in 
the projective geometry ${\rm PG}(U_0)$ may be lifted to create
generically irregular structures known as {\em subspace arrangements}.
These are more general configurations than the better known 
{\it hyperplane arrangements} studied by Bj\"orner~\cite{Bjorner} and others.
This can help to break symmetries, allowing us to see differences
in bimaps that have no other obvious distinctions.  

Let us define more carefully the objects we use.
An {\em (affine) subspace arrangement} is a set of subspaces of a vector space.  
We define the {\em idealizer arrangements} of $U_*=(U_2, U_1, U_0,\circ)$  
as follows:
\begin{align*}
	\mathcal{A}_2 & = \{ \lambda(P) \colon P\in {\rm PG}(U_0) \} &
	\mathcal{A}_1 & = \{ \rho(P) \colon P\in {\rm PG}(U_0) \}.
\end{align*}
We can build other arrangements for $U_*$ by shuffling coordinates.
One way to discover $\Aut(U_*)$-invariant substructures
is to compute the intersection numbers of the arrangement 
up to some rank. This is essentially what is described as a 
{\em fingerprint} in~\cite{ELGOB}
and the algorithm described there is polynomial in $|U_0|$.  
One could, however, use more refined tools such as the
M\"obius function $\mu\colon {\rm PG}(U_0) \times {\rm PG}(U_0) \to \mathbb{Z}$, where
\begin{align*}
	\mu(P,Q) & = \left\{\begin{array}{cc} 
		1 & \lambda(P)=\lambda(Q),\\
		-\sum_{\lambda(P)\leq X<\lambda(Q)} \mu(\lambda(P),X) & \lambda(P)<\lambda(Q),\\
		0 & \mbox{otherwise}.\end{array}\right.
\end{align*}

%%%%%
\section{Implementation and performance} 
\label{sec:imp}

We have implemented our algorithms in {\sc Magma}. 
The implementation utilizes various packages---all publicly available on GitHub---that have been
developed by the authors and their collaborators~\cite{git:tensor}.

The examples in Sections~\ref{sec:examples} and~\ref{sec:ex-rev} show how our local 
heuristics expose global structures. 
Here we briefly report on experiments with our 
% computer 
implementation that 
demonstrate their impact.

%%%
\subsection{Twisted Heisenberg groups}
\label{subsec:twisted}
Fix a prime $p$ and odd integer $k>1$. 
Put $q:=p^k$, $A:=\Bbb{F}_q$, and choose $\sigma\in{\rm Gal}(\Bbb{F}_q)$.
Define a product $\cdot$ on $A$, where 
$x\cdot y=xy+ix^{\sigma}y^{\sigma^2}$ for $x,y\in A$.
This turns $A$ into
a {\em twisted Albert algebra}, a nonassociative finite division algebra. Finally, define 
a bimap $\bullet\colon A^2\times A^2\bmto A$ by
\[
u\bullet v=u \left[ \begin{array}{rr} 0 & 1 \\ -1 & 0 \end{array} \right] v^{\dagger}.
\]
Over the ground field $\Bbb{F}_p$, this gives a bimap 
$\Bbb{F}_p^{2k}\times\Bbb{F}_p^{2k}\bmto\Bbb{F}_p^k$,
which can be encoded for computation as a system of $k$ alternating forms 
of degree $2k$. These bimaps
arise from the {\em twisted Heisenberg groups}
\[
H(A)=\left\{  \left[ \begin{array}{ccc} 1 & \alpha & \gamma \\ 0 & 1 & \beta \\ 0 & 0 & 1 \end{array} \right]  
\colon \alpha,\beta,\gamma\in A   \right\}.
\]
For various choices $1<e\leq k$, we now compose $\bullet$ with random 
$\varphi\in\Hom(\Bbb{F}_{p^k},\Bbb{F}_{p^e})$
to obtain bimaps $\circ\colon \Bbb{F}_p^{2k}\times\Bbb{F}_p^{2k}\bmto\Bbb{F}_p^e$, where 
$u\circ v:=u\bullet^{\varphi}v=(u\bullet v)\varphi$. The objective is to construct the group $\pseudo(\circ)$.
The default is to use the algorithm outlined in Section~\ref{sec:pseudo} which 
requires us to list $\GL(e,p)$.
An alternative is to build the projective geometry ${\rm PG}(e,p)$ and label its 
points and lines, as described in Section~\ref{sec:local}. 

Table \ref{table1}
records information from
experiments with our implementation. For different choices of $p,k$ and $e$, it records: 
the numbers of points and lines in ${\rm PG}(e,p)$; 
the approximate order of $\GL(e,p)$; 
the order of the subgroup $\Omega$ of $\GL(e,p)$ preserving the point and line labels; 
and finally the order
of the subgroup $\Psi$ of $\Omega$ that lifts to pseudo-isometries of $\circ$. 
For each choice of 
parameters, we performed 10 random trials and recorded the 
$|\Omega|$ and $|\Psi|$ that occurred most often. 
(In fact, these numbers were the same for all runs.)

\begin{table}[htb]
\begin{center}
\begin{tabular}{||c|c|c||c|c|c|c|c||}
\hline\hline
~~$p$~~ & ~~$k$~~ & ~~$e$~~ & ~{~\# points~} & ~{~\# lines}~ 
& ~~$|\GL(e,p)|$~~ &  ~~$|\Omega|$~~ & ~~$|\Psi|$~~ \\
\hline\hline
$3$ & $5$ & $3$ & 13  & 13 & $\approx 10^4$ & 16 & 2   \\ \hline
$3$ & $5$ & $4$ & 40  & 130 & $\approx 10^8$  & 10 & 10  \\ \hline
$3$ & $6$ & $4$ &  40 &  130 & $\approx 10^8$ & 1 & 1  \\ \hline
$5$ & $5$ & $3$ & 31  & 31 & $\approx 10^7$ & 1 &  1 \\ \hline
$5$ & $5$ & $4$ & 156  & 806 & $\approx 10^{11}$  & 20  & 20  \\ \hline \hline
\end{tabular}
\end{center}
\caption{Experiments for the twisted Heisenberg groups}\label{table1}
\end{table}

We remark that other weaker (but more efficiently computed) invariants 
can be used to label the lines in our
projective geometries. Often these suffice to discover structures 
that must be preserved by pseudo-isometries (or by autotopisms in more general applications). 
However, the weaker invariants do not distinguish quotients of twisted Heisenberg groups: 
all resulting line labels are identical.
But, by using the invariants
described in Section~\ref{sec:line-labels}, the problem breaks completely.

%%%
\subsection{Random nilpotent Lie algebras of class 2}
\label{subsec:random}
If $U_*=(U_2,U_1,U_0,\circ)$ is a generic bimap---which we loosely define to be one
specified by an arbitrary selection of structure constants---then the group induced by
$\Aut(U_*)$ on $U_0$ is usually very small, and is often trivial. In such cases, 
local invariants (even the weaker ones) almost always cut down to an overgroup 
containing the correct group as a subgroup of very small index.

In Table \ref{table2}, 
we mimic the experiments and reporting of Section~\ref{subsec:twisted}, but this time
we generate alternating bimaps 
$\circ\colon \Bbb{F}_p^d\times\Bbb{F}_p^d\bmto \Bbb{F}_{p}^e$
by selecting $e$ random skew-symmetric $d\times d$ matrices 
over $\Bbb{F}_p$.
Once again, we select the most commonly occurring $|\Omega|$ and $|\Psi|$ from 10 runs 
with each choice of parameters $(d,p,e)$.

\begin{table}[htb]
\begin{center}
\begin{tabular}{||c|c|c||r|r|c|c|c||}
\hline\hline
~~$d$~~ & ~~$p$~~ & ~~$e$~~ & {~\# points}~ & ~{~\# lines}~  & ~~$|\GL(e,p)|$~~ &  ~~$|\Omega|$~~ & ~~$|\Psi|$~~ \\
\hline\hline
$10$ & $3$ & $3$ & 13  & 13 & $\approx 10^4$ & 1 & 1   \\ \hline
$20$ & $3$ & $3$ &  13 & 13 & $\approx 10^4$  & 1 & 1  \\ \hline
$20$ & $3$ & $4$ &  40 &  130 & $\approx 10^8$ & 1 & 1  \\ \hline
$10$ & $3$ & $5$ & 121  & 1210 & $\approx 10^{11}$ & 1 &  1 \\ \hline
$20$ & $5$ & $3$ & 31  & 31 & $\approx 10^{6}$  & 1  & 1  \\ \hline
$10$ & $5$ & $4$ & 156  & 806 & $\approx 10^{11}$  & 1  & 1  \\ \hline
$10$ & $5$ & $5$ & 781  & 20306 & $\approx 10^{17}$  & 1  & 1  \\ \hline \hline
\end{tabular}
\end{center}
\caption{Experiments for Lie algebras of class 2}\label{table2}
\end{table}

These tests suggest that, for generic bimaps, $\pseudo(\circ)$ acts trivially on 
its codomain.
The variation of the dimension, $d$, of the domain space has little impact on the outcome,
but increasing $e$ introduces more constraints and therefore makes it increasingly 
likely that $\pseudo(\circ)$ acts trivially. As we see, the local invariants 
usually detect when this is the case.

%%%%%
\subsection{Do the heuristics always work?} 
The striking practical success of the local invariants raises the obvious question 
of whether we can strengthen 
the theoretical complexity of Theorem~\ref{thm:main}. 
If we label just points and lines, the answer is no: there exist alternating 
bimaps $U_*$ for which all 
points and lines are labeled identically, yet the group induced by $\pseudo(U_*)$ on its
codomain $U_0$ is a proper subgroup of $\Aut(U_0)$; one such example  appears in 
\cite{Verardi}.

%%%%%%%%%%%
%------ BIBLIOGRAPHY ------%

\end{document}